\documentclass[11pt, reqno]{amsart}

\usepackage{amsmath,amsfonts,amssymb,mathrsfs}
\usepackage{mathrsfs,extpfeil}
\usepackage{indentfirst, latexsym, enumerate}

\usepackage{graphicx,epsfig,subcaption,float}
\usepackage{cite}
\usepackage{colortbl,bm} 

\usepackage{hyperref}
\hypersetup{
	hyperindex=true
}
\hypersetup{hidelinks}
\usepackage{float}
\allowdisplaybreaks[4]

\title[Emergent behaviors of SCKM and its graph limit]{Emergent behaviors of the singular continuum Kuramoto model and Its Graph Limit}

\author[L. Chen]{Li Chen}
\address[Li Chen]{\newline School of Business Informatics and Mathematics \newline
University of Mannheim, Mannheim 68159, Germany}
\email{li.chen@uni-mannheim.de}

\author[S.-Y. Ha]{Seung-Yeal Ha}
\address[Seung-Yeal Ha]{\newline Department of Mathematical Sciences and Research Institute of Mathematics \newline  
Seoul National University, Seoul 08826, Republic of Korea}
\email{syha@snu.ac.kr}

\author[X. Wang]{Xinyu Wang}
\address[Xinyu Wang]{\newline Department of Mathematical Sciences \newline  
Seoul National University, Seoul 08826, Republic of Korea}
\email{wangxinyu97@snu.ac.kr}

\author[V. Zhidkova]{Valeriia Zhidkova}
\address[Valeriia Zhidkova]{\newline School of Business Informatics and Mathematics  \newline
University of Mannheim, Mannheim 68159, Germany}
\email{valeriia.zhidkova@uni-mannheim.de}

\newtheorem{theorem}{Theorem}[section]
\newtheorem{lemma}{Lemma}[section]
\newtheorem{corollary}{Corollary}[section]
\newtheorem{proposition}{Proposition}[section]

\newtheorem{remark}{Remark}[section]

\newcommand{\bbr}{\mathbb R}

\def\charf {\mbox{{\text 1}\kern-.24em {\text l}}}

\newcommand*\di{\mathop{}\!\mathrm{d}}
\allowdisplaybreaks[4]

\begin{document}

\date{\today}


\thanks{\textbf{Acknowledgement.}
The work of S.-Y. Ha is supported by National Research Foundation(NRF) grant funded by the Korea government(MIST) (RS-2025-00514472), and the work of X. Wang was  partially supported by the Natural Science Foundation of China (grants 123B2003), the China Postdoctoral Science Foundation (grants 2025M774290), and Heilongjiang Province Postdoctoral Funding (grants  LBH-Z24167). The work of L. Chen is partially supported by Deutsche Forschungsgemeinschaft (DFG, German Research Foundation – 547277619) and the National Natural Science Foundation of China (12171218). The work of V. Zhidkova is partially supported by Deutsche Forschungsgemeinschaft (DFG, German Research Foundation – 547277619).}

\begin{abstract}
We study the emergent dynamics of the singular continuum Kuramoto model (in short, SCKM) and its graph limit. The SCKM takes the form of an integro-differential equation exhibiting two types of nonlocal singularities: a nonlocal singular interaction weight and a nonlocal singular alignment force. The natural frequency function determines the emergent dynamics of the SCKM, and we emphasize that singularity plays a crucial role in the occurrence of sticking phenomena. For the identical natural frequency function, we derive the complete phase synchronization in finite time under a suitable set of conditions for system parameters and initial data. In contrast, for a nonidentical natural frequency function, we show the emergence of practical synchronization, meaning that the phase diameter is proportional to the inverse of coupling strength asymptotically.  Furthermore, we rigorously establish a graph limit from the singular Kuramoto model with a finite system size to the SCKM. We also provide several numerical simulations to illustrate our theoretical results.
\end{abstract}
\maketitle \centerline{\date}

\section{Introduction} \label{sec:1}
\setcounter{equation}{0}
Synchronization is a kind of coherent collective behavior appearing in oscillatory systems, where individual units adjust their rhythms (phases and frequencies) through the mutual couplings or common forcing. This collective phenomenon is ubiquitous in biology, chemistry, physics, and social systems, with canonical examples including the synchronous flashing of fireflies \cite{B-B1966}, coordinated oscillations in insulin secretion and glucose infusion \cite{S-K1995}, and pacemaker synchronization in the heart \cite{Winfree1977}. The first modern mathematical theory originated from the phase-reduction framework of Winfree \cite{Winfree1977}, and  then Yoshiki Kuramoto proposed the Kuramoto model to describe the synchronization of weakly coupled oscillators \cite{Kuramoto1975}. We refer to the survey articles \cite{A-B2005, A-B2019, B-P2012, D-B2014, HaKo2016, T-T1998, B-H2010} for the crash introduction of collective behaviors. 

In this paper, we focus on the emergent dynamics and the graph limit of Kuramoto oscillators under the double singular effects in communication and coupling. To set up the stage, let $(\theta_i, \nu_i)$ be the time-dependent phase and stationary natural frequency of  the $i$-th Kuramoto oscillator. Then, their dynamics is governed by the Cauchy problem for the Kuramoto model:
 \begin{equation} \label{A-1}
\begin{cases}
\displaystyle \dot{\theta}_i(t) = \nu_i + \frac{\kappa}{N} \sum_{k=1}^{N} \sin(\theta_{k}(t) - \theta_i(t)), \quad t>0, \\
\displaystyle \theta_i(0) = \theta_i^{\mathrm{in}}, \quad i\in[N]:=\{1,2,...,N\},
\end{cases}
 \end{equation}
where $\kappa>0$ is the nonnegative coupling strength. \newline 
 
 Note that the emergent dynamics of the Cauchy problem \eqref{A-1} has been extensively studied in a series of works \cite{C-S2023, D-B2012, D-B2011, HaKim2016}, and we also refer to the survey papers \cite{Erm-1985,H-R,D-B2014,A-B2005}.  When the number of particles tends to infinity, how to describe the evolution of the limiting equation is an important question in statistical physics. Broadly speaking, several frameworks are used to describe emergent dynamics in infinite particle systems. A classical approach concerns mean-field type interactions associated with all-to-all coupling, where one derives a kinetic equation via a mean-field limit \cite{Poyato2019,B-C-M,w6,w9, Lucon2014}. A second approach treats systems on dense graphs, where graph limit techniques lead to nonlocal continuum equations \cite{ChoHa2023,B-D2022,KoHa2025,MWX2025}. This framework is well-suited to dense deterministic graphs, but does not capture more general heterogeneous or measure-valued limits. To address such situations, extended mean-field formulations have been developed \cite{KV2018,KuehnXu2022, JPS2024}. Finally, one may study dynamics directly as infinite systems of ODEs posed on infinite graphs \cite{Bra-2019,w1,w5}, without passing through a continuum approximation. In this work, we consider the second approximation methodology by focusing on the emergent dynamics and graph limit of singular continuum Kuramoto model (in short, SCKM) as follows:
\begin{equation} \label{A-2}
	\begin{cases}
		\partial_t \theta(t, x) = \nu(x) + \kappa \displaystyle\int_{\Omega} \psi(x,y) h( \theta(t, y) - \theta(t, x)) \di y, \quad t > 0,~~x \in \Omega, \\
		\theta(0, x) = \theta^{\mathrm{in}}(x), 
	\end{cases}
\end{equation}
where $\psi:~\Omega \times \Omega \to [0, \infty)$ is a singular interaction weight and $h:~{\mathbb R} \to {\mathbb R}$ is a singular alignment force kernel. They are given by the following explicit forms:
\begin{align}
	\label{A-3}
	\psi (x,y) :=\frac{1}{|x-y|^\beta}  \quad (0\le \beta<1), \quad h(\theta) := 
	\begin{cases}
	\displaystyle	\frac{\sin (\theta)}{|\theta|_o^\alpha} \ &\mathrm{for} \ \theta \neq 0 \quad  (0\le \alpha<1), \\
		\displaystyle	0 \ &\mathrm{for} \ \theta=0,
	\end{cases}
\end{align}
where 
\begin{align*}
|\theta|_o := |\bar \theta| \quad \mathrm{for} \ \bar \theta \equiv \theta \mod 2\pi \quad \mathrm{and} \ \bar \theta \in (-\pi, \pi].
\end{align*}
We notice that $h$ is $2\pi$-periodic on $\bbr$. Throughout this paper, we consider the following two domains for $\Omega$:
\[ \mbox{Either}~\Omega = \bbr^d \quad \mbox{or} \quad \Omega = [0, 1]. \]
In this work, we address the following three questions for the SCKM \eqref{A-2}:  \newline
\begin{itemize}
\item
(Q1):~ What will happen to the SCKM on $\Omega = \bbr^d$ if we set $\psi$ to be the approximation of Dirac delta distribution?
\vspace{0.1cm}
\item
(Q2):~Under what conditions on system $\nu, \psi$ and $h$ and initial data $\theta^{\mathrm{in}}$, can we show the emergence of collective behaviors for the SCKM on $\Omega = [0,1]$?
\vspace{0.1cm}
\item
(Q3):~Can we approximate the SCKM on $\Omega = [0,1]$ via the finite Kuramoto lattice model using the juxtaposition of piecewise functions? 
\end{itemize}

\vspace{0.1cm}

The purpose of this paper is to deal with the aforementioned questions. More precisely, our main results are three-fold. First, we show that the solution to \eqref{A-2} with $\psi_\varepsilon (x,y) = \eta_\varepsilon(x-y)$ converges to the solution of the linear equation:
\begin{align*}
	\begin{cases}
		\partial_t \theta(t, x) = \nu(x),  & t > 0, \quad x \in \bbr^d, \\
		\theta(0, x) = \theta^{\mathrm{in}}(x),
	\end{cases}
\end{align*}
as $\varepsilon \rightarrow 0$ (see Proposition \ref{P3.1}). 

Second, we establish the global well-posedness of \eqref{A-2} on $\Omega = [0,1]$ in which Lipschitz continuity is deprived because of the presence of singularities (see Theorem \ref{T3.1} and Remark \ref{R3.1}). We notice that the global well-posedness result easily extends to a bounded domain $\Omega \subset \bbr^d$. Then, we focus on the emergent dynamics of the SCKM \eqref{A-2} with constant natural frequency function and nonidentical natural frequencies, respectively. Compared to references \cite{KoHa2025, Poyato2019, P-P-J2021}, the main technical difficulty lies in the fact that \eqref{A-2} contains the double nonlocal singularities that depend on the position $x\in \Omega$ and phase $\theta$. In what follows, we briefly summarize emergent dynamics results. Consider the system parameters and initial data satisfying the following set of conditions:
\[  \alpha, \beta \in (0,1), \quad \theta^{\mathrm{in}} \in L^2(\Omega), \quad \nu \equiv 0, \quad 0<\mathcal{D}(\theta^{\mathrm{in}})<\pi, \]
where $\mathcal{D}(\theta)$ is the functional measuring phase coherence: 
\[  \mathcal{D}(\theta(t)) = \sup_{x,y\in \Omega} |\theta(t,x) - \theta(t,y)|. \]
Then, complete synchronization emerges in finite-time:
\begin{align*}
	\mathcal{D}(\theta(t)) \leq \left(\mathcal{D}(\theta^{\mathrm{in}})^{\alpha} - \alpha \kappa \frac{\sin\big( \mathcal{D}(\theta^\mathrm{in}) \big)}{\mathcal{D}(\theta^\mathrm{in})}  \frac{2-2^{\beta}}{1-\beta} t \right)^{\frac{1}{\alpha}}.
\end{align*}
This implies that there exists a positive constant $t_c$ such that 
\[  \lim_{t \to t_c-} \mathcal{D}(\theta(t)) = 0. \]
We refer to Theorem \ref{T3.2} for details. In contrast, for nonidentical natural frequencies with $\nu \in B(\Omega)$, where $B(\Omega)$ denotes the set of bounded functions on $\Omega$, we have the following practical synchronization:
\[ \lim_{\kappa \to \infty} \limsup_{t\to \infty} \mathcal{D}(\theta(t)) = 0.\]
See Theorem \ref{T3.3} for details. 

Finally, we establish the graph limit from the finite Kuramoto system to the corresponding continuum limit SCKM \eqref{A-2} on $\Omega = [0,1]$. The main difficulty lies in the fact that H\"older continuity of $h$ does not suffice to do that. For this reason, we are required to consider the angles modulo $2\pi$, which allows us to impose one-sided Lipschitz continuity on $-h$. Then, we can establish the finite-time graph limit for a heterogeneous system and the uniform-in-time graph limit for a homogeneous system. We refer to Theorem \ref{T4.1} for details. \newline

The rest of this paper is organized as follows. In Section \ref{sec:2}, we study preparatory lemmas of the singular communication weight function and singular interaction function and Lebesgue functions. In Section \ref{sec:3}, we study the well-posedness and emergent dynamics of the SCKM. In Section \ref{sec:4}, we present a rigorous verification of the graph limit from the lattice model to the corresponding continuum equation. In Section \ref{sec:5}, we provide several numerical simulation results. Finally, Section \ref{sec:6} is devoted to a brief summary of main results and remaining issues for a future work. 

\vspace{0.5cm}

\noindent {\bf Notations:}~Let $\nu = \nu (x),~\theta = \theta(t,x)$ and $\psi = \psi(x,y)$ be measurable functions on $\Omega, \mathbb{R}_+ \times \Omega$ and $\Omega \times \Omega$, respectively. Then, for any $p \in [1,\infty],~t \in \bbr_+,~x \in \Omega$, 
\begin{align*}
&\| \nu \|_p := \| \nu \|_{L^p (\Omega)}, \quad \| \theta (t) \|_p := \| \theta (t) \|_{L^p (\Omega)}, \\ 
&\| \psi \|_p := \| \psi \|_{L^p (\Omega^2)}, \quad \| \psi (x,\cdot) \|_p := \| \psi (x,\cdot) \|_{L^p (\Omega)}.
\end{align*}
We define natural frequency diameter and shift of function $f: \bbr^d \rightarrow \bbr$ as follows:
\begin{align*}
\mathcal{D}(\nu) = \sup_{x,y \in \Omega} |\nu(x) - \nu(y)|, \quad (\tau_v f)(x) = f(x- v), \quad x\in \bbr^d,\quad v\in \bbr^d.
\end{align*}
If $f$ takes values on $D\subset \bbr^d$, then we can define the zero extension $\tilde{f}$ of $f$ as follows:
\begin{align*}
\tilde{f} (x)= 
\begin{cases}
f(x) , \quad & x\in D, \\
0,  &\mathrm{otherwise}.
\end{cases}
\end{align*}
Then we use the following simplified notation:
\begin{align*}
\int_D |f(x-v) - f(x)| \di x := \int_D |\tilde{f}(x-v) - \tilde{f}(x)| \di x.
\end{align*}
We denote 
\begin{align*}
\bbr_+ := [0, \infty) \quad \mathrm{and} \quad  B(\Omega) = \lbrace f: \Omega \to \bbr \ \mathrm{is} \ \mathrm{bounded} \rbrace.
\end{align*}
Throughout the paper, for a set $\Omega$, we denote by $|\Omega|$ the Lebesgue measure of the set $\Omega$. 

\section{Preliminaries} \label{sec:2}
\setcounter{equation}{0}
In this section, we study several a priori estimates to be used for later sections. More precisely, we deal with the translation invariance, properties of singular interaction function and singular communication weight function. We also provide preparatory lemmas for the $L^p$ functions.
\subsection{Translation invariance} \label{sec:2.1}
In this subsection, we study the translation invariance of the SCKM \eqref{A-2}. For a bounded domain with $ |\Omega| < \infty$, we define averages for phase and natural frequency and their fluctuations around the averages:~for $(t,x) \in \bbr_+ \times \Omega$, 
\begin{align}
\begin{aligned} \label{B-1}
& {\bar \theta} (t) := \frac{1}{|\Omega|} \int_\Omega \theta (t,x) \di x, \quad {\bar \nu} :=  \frac{1}{|\Omega|}  \int_\Omega \nu (x) \di x, \\
& \tilde{\theta} (t,x) := \theta(t,x) - {\bar \nu}  t, \quad \tilde{\nu} (x) :=  \nu(x) - {\bar \nu}.
\end{aligned}
\end{align}
Then, it is easy to see that 
\[
\int_{\Omega}  \tilde{\nu} (x)  \di x = 0. 
\]
\begin{lemma} \label{L2.1}
Let $\theta = \theta(t,x)$ be a global solution to \eqref{A-2} on a bounded domain $\Omega$. Then, it holds
\[
\partial_t \tilde{\theta} (t,x) =  {\tilde \nu}(x) +  \kappa \displaystyle\int_{\Omega} \psi(x,y) h( {\tilde \theta}(t, y) - {\tilde \theta}(t, x)) \di y, \quad 
\frac{\di}{\di t} \bar \theta (t) = {\bar \nu}, \quad t > 0, \ x\in \Omega.
\]
\end{lemma}
\begin{proof}
(i)~First, we note that $\bar{\nu}$ is constant. Then we use \eqref{B-1} to find 
\[  \theta(t, y) - \theta(t, x) =  \Big( \tilde{\theta} (t,y) + {\bar \nu}  t \Big) - \Big( \tilde{\theta} (t,x) + {\bar \nu}  t \Big) =  \tilde{\theta} (t,y) - \tilde{\theta} (t,x).        \]
This implies 
\begin{align*}
\begin{aligned}
\partial_t \tilde{\theta} (t,x) &= \partial_t \theta(t,x) - {\bar \nu} =  \nu(x) - \bar{\nu} + \kappa \displaystyle\int_{\Omega} \psi(x,y) h( \theta(t, y) - \theta(t, x)) \di y  \\
& = {\tilde \nu}(x)  +  \kappa \displaystyle\int_{\Omega} \psi(x,y) h( {\tilde \theta}(t, y) - {\tilde \theta}(t, x)) \di y.
\end{aligned}
\end{align*}
(ii)~Note that 
\begin{align}
\begin{aligned} \label{B-2}
\frac{\di}{\di t} \bar \theta (t) &= \frac{1}{|\Omega|} \int_\Omega \partial_t \theta(t,x)  \di x \\
& = \frac{1}{|\Omega|}\int_\Omega \nu (x)  \di x + \underbrace{\frac{\kappa}{|\Omega|} \iint_{\Omega^2} \psi(x, y) \frac{\sin \big( \theta(t, y) - \theta(t, x) \big)}{|\theta(t, y) - \theta(t, x)|_o^\alpha} \di y \di x}_{=:{\mathcal I}_1}.
\end{aligned}
\end{align}
Since $\psi(x,y)$ and $|\theta(t,x) - \theta(t,y)|_o^\alpha$ are symmetric and the $\sin$-function is odd, the term ${\mathcal I}_{1}$ vanishes:
\begin{align}
\begin{aligned} \label{B-3}
{\mathcal I}_{1} &= \frac{\kappa}{|\Omega|} \iint_{\Omega^2}  \psi(x, y) \frac{\sin \big( \theta(t, y) - \theta(t, x) \big)}{|\theta(t, y) - \theta(t, x)|_o^\alpha} \di y \di x \\
&= -\frac{\kappa}{|\Omega|} \iint_{\Omega^2} \psi(x, y) \frac{\sin \big( \theta(t, y) - \theta(t, x) \big)}{|\theta(t, y) - \theta(t, x)|_o^\alpha} \di y \di x = -{\mathcal I}_{1}.
\end{aligned}
\end{align} 
Finally, we combine \eqref{B-2} and \eqref{B-3} to find the desired second estimate.
\end{proof}

\subsection{Singular coupling function}\label{sec:2.2}
In this subsection, we study several properties of the singular interaction kernel $h$ in \eqref{A-3} and its regularized version 
\begin{align} \label{B-3-1}
h_\varepsilon (\theta) = \frac{\sin(\theta)}{|\theta|_o^\alpha+ \varepsilon}, \quad \theta \in \bbr, \ \varepsilon >0.
\end{align}

\begin{lemma} \label{L2.2}
For $q\in [0,1]$ and $u,v\in [0, \infty)$, we have
\begin{align*}
u^q + v^q \geq (u+v)^q.
\end{align*}
\end{lemma}
\begin{proof}
See Appendix \ref{App-0}.
\end{proof}
\begin{lemma}\label{L2.3}
For $\alpha \in [0, 1),$ the coupling function $h$ is $(1-\alpha)$-H\"older continuous on $\mathbb{R}$:
\[\sup\limits_{\theta_1 \neq \theta_2\in \mathbb{R}}\frac{|h(\theta_1) - h(\theta_2)|} {|\theta_1-\theta_2|^{1-\alpha}}<\infty.\] 
Moreover, the function $h$ is bounded on $\bbr$:
\[ \sup_{\theta\in\bbr} |h(\theta)| \leq 1.\]
\end{lemma}
\begin{proof} 
\noindent (i)~The function $h$ is clearly $2\pi$-periodic and smooth away from $\theta = 2\pi k, \ k\in \mathbb{Z}$. Therefore we prove that $h$ is $(1-\alpha)$-H\"older continuous on $\left[ -\frac{1}{2}, \frac{1}{2} \right]$. On this set it holds $|\theta|_o = |\theta|$. We split the proof into three steps. \newline

\noindent $\bullet$~Step A (Preparatory step):~We first rewrite $h$ in \eqref{A-3} as 
\begin{align*}
h(\theta) = \frac{\sin (\theta)}{|\theta|} |\theta|^{1-\alpha}, \quad \theta \neq 0.
\end{align*}
We claim that for all $\theta_1,\theta_2 \in \mathbb{R}$ and $q\in (0,1]$, we have
\begin{align} \label{B-4}
\Big ||\theta_1|^q - |\theta_2|^q \Big | \leq |\theta_1-\theta_2|^q.
\end{align}
{\it Proof of \eqref{B-4}}:~Let $u \geq v \geq 0$. Then, by Lemma \ref{L2.2}, we have
\begin{equation} \label{B-5}
u^q = (u-v +v)^q \leq (u-v)^q + v^q, \quad \mbox{i.e.,} \quad (u-v)^q \geq u^q - v^q.
\end{equation}
Now, without loss of generality, we may assume $|\theta_1|\geq |\theta_2|$, and we set 
\[ u:=|\theta_1| \quad \mbox{and} \quad  v:=|\theta_2|. \]
Then, we substitute $u$ and $v$ into \eqref{B-5} to find 
\begin{align*}
|\theta_1-\theta_2|^q \geq \Big ||\theta_1|-|\theta_2| \Big |^q \geq |\theta_1|^q - |\theta_2|^q = \Big||\theta_1|^q - |\theta_2|^q \Big |.
\end{align*}
\noindent $\bullet$~Step B:~It is clear that the function $\phi(\theta) = \frac{\sin (\theta)}{|\theta|}$ which is continuously extended by $\phi(0) = 1$ is Lipschitz continuous on $\mathbb{R}$ with Lipschitz constant $L_{\phi}>0$.
\vspace{0.2cm}

\noindent $\bullet$~Step C:~Finally, we combine the results of Step A and Step B to derive
\begin{align}
\begin{aligned} \label{B-7}
|h(\theta_1) - h(\theta_2)| = &\left| \frac{\sin (\theta_1)}{|\theta_1|} |\theta_1|^{1-\alpha} - \frac{\sin (\theta_2)}{|\theta_2|} |\theta_2|^{1-\alpha} \right| \\
=& \left| \frac{\sin (\theta_1)}{|\theta_1|} |\theta_1|^{1-\alpha} - \frac{\sin (\theta_1)}{|\theta_1|} |\theta_2|^{1-\alpha} + \frac{\sin (\theta_1)}{|\theta_1|} |\theta_2|^{1-\alpha} - \frac{\sin (\theta_2)}{|\theta_2|}|\theta_2|^{1-\alpha} \right| \\
\leq & \left| \frac{\sin (\theta_1)}{|\theta_1|} \right| | |\theta_1|^{1-\alpha} - |\theta_2|^{1-\alpha}| + |\theta_2|^{1-\alpha} \left|  \frac{\sin (\theta_1)}{|\theta_1|} - \frac{\sin (\theta_2)}{|\theta_2|} \right| \\
\leq & |\theta_1-\theta_2|^{1-\alpha} + L_{\phi}|\theta_2|^{1-\alpha} |\theta_1-\theta_2|.
\end{aligned}
\end{align}
Here, we used $\left|\frac{\sin (\theta_1)}{|\theta_1|}\right|\le1$ and \eqref{B-4} in the last inequality. Since $|\theta_1-\theta_2|\leq 1$, we use $|\theta_1-\theta_2|\le|\theta_1-\theta_2|^{1-\alpha}$ and \eqref{B-7} to see
\begin{align*}
|h(\theta_1) - h(\theta_2)|&\leq |\theta_1-\theta_2|^{1-\alpha} + L_{\phi}|\theta_2|^{1-\alpha} |\theta_1-\theta_2| \leq (1+L_{\phi}|\theta_2|^{1-\alpha}) |\theta_1-\theta_2|^{1-\alpha} \\
&\leq \left(1+ 2^{\alpha -1}L_{\phi}\right) |\theta_1-\theta_2|^{1-\alpha} \leq C(L_{\phi},\alpha)|\theta_1-\theta_2|^{1-\alpha}.
\end{align*}
This gives us the desired estimate. \\

\noindent (ii)~Let $\theta \in \bbr$ and $\bar \theta \in (-\pi, \pi]$ with $\bar \theta = \theta + 2\pi k, \ k \in \mathbb{Z}$. Then for $|\bar \theta|\leq 1$ it holds
\begin{align*}
|h(\theta)| = \frac{|\sin (\bar \theta)|}{|\bar \theta|^\alpha} \leq |\bar \theta|^{1-\alpha} \leq 1
\end{align*}
and for $|\bar \theta|>1$ it holds $|\bar \theta|^\alpha >1$ and therefore
\begin{align*}
|h(\theta)| = \frac{|\sin (\bar \theta)|}{|\bar \theta|^\alpha} \leq \frac{1}{|\bar \theta|^\alpha} <1.
\end{align*}
Therefore
\begin{align*}
|h(\theta)| \leq 1 \ \mathrm{for} \ \theta \in \bbr. 
\end{align*}
\end{proof}

\begin{remark}
It is obvious that for every $\varepsilon>0$ it holds
\begin{align*}
\sup_{\theta\in \bbr}|h_\varepsilon(\theta)| \leq \sup_{\theta\in \bbr}|h(\theta)| \leq 1.
\end{align*}
\end{remark}

\begin{lemma} \label{L2.4-1}
The function $h_\varepsilon$ is H\"older continuous with H\"older constant $C_h$ which does not depend on $\varepsilon>0$. 
\end{lemma}
\begin{proof} As in the proof of Lemma \ref{L2.3}, we only consider the set $\left[-\frac{1}{2}, \frac{1}{2} \right]$. Let $\varepsilon>0$ and assume $\theta_1,\theta_2 \in \left[-\frac{1}{2}, \frac{1}{2} \right]$. We consider two cases. \newline \\
\noindent $\bullet$~Case A ($\max\lbrace |\theta_1|,|\theta_2| \rbrace \leq 2|\theta_1-\theta_2|$): 
\begin{align*}
\left| h_\varepsilon(\theta_1)-h_\varepsilon(\theta_2) \right| \leq \left| h_\varepsilon(\theta_1) \right|+ \left| h_\varepsilon(\theta_2) \right| \leq |\theta_1|^{1-\alpha}+ |\theta_2|^{1-\alpha} \leq 2^{2-\alpha} |\theta_1-\theta_2|^{1-\alpha}.
\end{align*}
\noindent $\bullet$~Case B ($\max\lbrace |\theta_1|,|\theta_2| \rbrace > 2|\theta_1-\theta_2|$): Assume $|\theta_1|\leq |\theta_2|$. Then it holds
\begin{align*}
|\theta_2|-|\theta_1| \leq |\theta_1-\theta_2| \quad \Rightarrow \quad |\theta_1| \geq |\theta_2| - |\theta_1-\theta_2| >\frac{|\theta_2|}{2}.
\end{align*} 
There exists a $\xi$ between $\theta_1$ and $\theta_2$ such that
\begin{align} \label{B-8-1}
\left| h_\varepsilon(\theta_1)-h_\varepsilon(\theta_2) \right| = |h_\varepsilon'(\xi)| |\theta_1-\theta_2|.
\end{align}
It holds
\begin{align*}
|h_\varepsilon'(\xi) |
= \left| \frac{\cos(\xi)}{|\xi|^\alpha+\varepsilon} - \frac{\alpha \mathrm{sgn}(\xi)\sin(\xi) |\xi|^{\alpha-1}}{(|\xi|^{\alpha}+\varepsilon)^2} \right| \leq (1+ \alpha) |\xi|^{-\alpha} \leq 2 |\xi|^{-\alpha}.
\end{align*}
Since $|\xi| > |\theta_1| > \frac{|\theta_2|}{2} > |\theta_1-\theta_2|$, it holds $|\xi|^{-\alpha} < |\theta_1-\theta_2|^{-\alpha}$. Combining this with \eqref{B-8-1} yields
\begin{align*}
\left| h_\varepsilon(\theta_1)-h_\varepsilon(\theta_2) \right| = |h_\varepsilon'(\xi)| |\theta_1-\theta_2| \leq 2 |\theta_1-\theta_2|^{1-\alpha}.
\end{align*}
Therefore $h_\varepsilon$ is $(1-\alpha)$-H\"older continuous on $\bbr$ with some H\"older constant $C_h$.
\end{proof}

We now impose a one-sided Lipschitz condition on $-h$ on $[-2\pi,2\pi]$, formulated as follows. \\
Let $\alpha\in[0,1)$ and let $h:[-2\pi,2\pi]\to\mathbb{R}$ be given.  Then, we choose $\bar h$ and $\tilde\theta\in(0,\tfrac{\pi}{2})$ such that
	\[
	\bar h \;=\; \max_{0<\theta<\pi} h(\theta),
	\qquad
	2\alpha \sin \tilde\theta \;=\; \tilde\theta \cos \tilde\theta.
	\]
	Define $\Delta,\Lambda:[-2\pi,2\pi]\to\mathbb{R}$ by
	\[
	\Delta(\theta):=
	\begin{cases}
		2\bar h- h(\theta), & \theta \in [-2\pi,\,-2\pi+\tilde\theta),\\[2pt]
		\bar h, & \theta \in [-2\pi+\tilde\theta,\,-\tilde\theta),\\[2pt]
		-\,h(\theta), & \theta \in [-\tilde\theta,\,\tilde\theta],\\[2pt]
		-\,\bar h, & \theta \in (\tilde\theta,\,2\pi-\tilde\theta],\\[2pt]
		-\,h(\theta)-2\bar h, & \theta \in (2\pi-\tilde\theta,\,2\pi],
	\end{cases}
	\qquad
	\Lambda(\theta):=
	\begin{cases}
		-\,2\bar h, & \theta \in [-2\pi,\,-2\pi+\tilde\theta),\\[2pt]
		-\,\bar h- h(\theta), & \theta \in [-2\pi+\tilde\theta,\,-\tilde\theta),\\[2pt]
		0, & \theta \in [-\tilde\theta,\,\tilde\theta],\\[2pt]
		\bar h- h(\theta), & \theta \in (\tilde\theta,\,2\pi-\tilde\theta],\\[2pt]
		2\bar h, & \theta \in (2\pi-\tilde\theta,\,2\pi].
	\end{cases}
	\]
\begin{lemma} \label{P4.1}
The following assertions hold:
	\begin{enumerate}
		\item[(i)] $\Delta$ is monotonically decreasing, and $\Lambda$ is Lipschitz continuous such that 
		\[
		-\,h(\theta)\;=\;\Delta(\theta)+\Lambda(\theta), \quad \theta\in[-2\pi,2\pi].
		\]
		\item[(ii)] The function $-h$ is one–sided Lipschitz continuous on $[-2\pi,2\pi]$: there exists $L_h>0$ such that for all $\theta_1,\theta_2\in[-2\pi,2\pi]$,
		\[
		\big((-h)(\theta_1)-(-h)(\theta_2)\big)\,(\theta_1-\theta_2)\ \le\ L_h\,(\theta_1-\theta_2)^2.
		\]
	\end{enumerate}
\end{lemma}
\begin{proof}
We refer to Lemma 2.17 in \cite{Poyato2019} for a detailed proof.
\end{proof}

\begin{lemma} \label{L2.6-1}
Let $\varepsilon>0$. The function $-h_\varepsilon$ is one-sided Lipschitz continuous on $[-2\pi,2\pi]$ with some Lipschitz constant $L$ independent of $\varepsilon$.
\end{lemma}
\begin{proof}
\noindent $\bullet$~Case A ($\theta_1, \theta_2 \in [-\pi, \pi]$):
There exists $\xi$ between $\theta_1$ and $\theta_2$ such that
\begin{align*}
((-h_\varepsilon)(\theta_1) - (-h_\varepsilon)(\theta_2))(\theta_1 - \theta_2) = (-h_\varepsilon)'(\xi) (\theta_1 - \theta_2)^2 \leq \max \left\lbrace \sup_{\theta\in [-\pi, \pi]}(-h_\varepsilon)', 1 \right\rbrace  (\theta_1 - \theta_2)^2.
\end{align*}
We compute
\begin{align*}
(-h_\varepsilon)'(\xi) 
= -\frac{\cos(\xi)}{|\xi|^\alpha+\varepsilon} + \frac{\alpha \mathrm{sgn}(\xi)\sin(\xi) |\xi|^{\alpha-1}}{(|\xi|^{\alpha}+\varepsilon)^2} \leq \left( \frac{\alpha |\sin(\xi)|}{|\xi|} -\cos(\xi) \right) \frac{1}{|\xi|^\alpha +\varepsilon}.
\end{align*} 
There exists $0<\delta<\frac{\pi}{2}$ independent of $\varepsilon$ such that $(-h_\varepsilon)'(\xi)\leq 0$ for $|\xi|<\delta$. For $|\xi|\geq \delta$ it holds
\begin{align*}
(-h_\varepsilon)' (\xi) \leq  (\alpha+1) \frac{1}{\delta^\alpha}.
\end{align*}
Therefore, $-h_\varepsilon$ is one-sided Lipschitz continuous on $[-\pi, \pi]$ with Lipschitz constant 
\begin{align*}
L:= \max \left\lbrace \frac{\alpha+1}{\delta^\alpha} , 1 \right\rbrace.
\end{align*}

\noindent $\bullet$~Case B ($\theta_1, \theta_2 \in [\pi, 2\pi]$ or $\theta_1, \theta_2 \in [-2\pi, -\pi]$): This case follows directly from Case A due to $2\pi$-periodicity of $-h_\varepsilon$. \newline

\noindent $\bullet$~Case C ($\theta_1 \in [\pi, 2\pi]$ and $\theta_2 \in [-\pi, \pi]$): Using the result of Case A, we get
\begin{align*}
&((-h_\varepsilon)(\theta_1) - (-h_\varepsilon)(\theta_2))(\theta_1 - \theta_2) \\
&\hspace{0.5cm}= ((-h_\varepsilon)(\theta_1) - (-h_\varepsilon)(\pi) + (-h_\varepsilon)(\pi)- (-h_\varepsilon)(\theta_2))(\theta_1 - \theta_2) \\
&\hspace{0.5cm}= ((-h_\varepsilon)(\theta_1) - (-h_\varepsilon)(\pi))(\theta_1-\pi) + ((-h_\varepsilon)(\theta_1) - (-h_\varepsilon)(\pi))(\pi -\theta_2) \\
&\hspace{0.7cm}+ ((-h_\varepsilon)(\pi)- (-h_\varepsilon)(\theta_2))(\theta_1 - \pi )+ ((-h_\varepsilon)(\pi)- (-h_\varepsilon)(\theta_2))(\pi-\theta_2) \\
&\hspace{0.5cm}\leq L (\theta_1-\pi)^2 + 2L(\theta_1- \pi)(\pi-\theta_2) + L (\pi - \theta_2)^2 \\
&\hspace{0.5cm} = L((\theta_1 - \pi) + (\pi - \theta_2))^2 = L(\theta_1 - \theta_2)^2.
\end{align*}

\noindent $\bullet$~Case D ($\theta_1 \in [-2\pi, -\pi]$, $\theta_2 \in [-\pi, \pi]$): This case can be shown in a similar way as Case C.
\end{proof}

\begin{lemma} \label{L2.4}
The following assertions hold.
\begin{enumerate}
\item[(i)]
For $\alpha \in (0, 1)$, the function $g(\theta) = \frac{\sin (\theta)}{\theta^{1+\alpha}}$ is nonincreasing on $(0,\pi)$. 
\vspace{0.1cm}
\item[(ii)]
The function $f(\theta) = \frac{\sin (\theta)}{\theta}$ is nonincreasing on $(0,\pi)$.
\end{enumerate}
\end{lemma}
\begin{proof} 
See Appendix \ref{App-0}.
\end{proof}
\begin{remark}\label{R2.1}
Even if the estimates in Lemma \ref{L2.4} are not immediate properties of the singular coupling function $h$, they allow us to obtain the following estimates for $h$:
\[
\theta_2 h(\theta_1) \leq \theta_1 h(\theta_2) \quad \mathrm{and} \quad h(\theta_2) \leq  \frac{\sin (\theta_1)}{\theta_1} \theta_2^{1-\alpha} \quad \mbox{for $\theta_1,\theta_2 \in (0,\pi)$ with $\theta_1\geq \theta_2$. }
\]

\end{remark}
\subsection{Singular communication weight function}\label{sec2.3}

For $\beta \in (0,1)$ the function $\psi (x,\cdot)$ is integrable over $\Omega = [0,1]$ for every fixed $x\in \Omega$:
\begin{align}\label{B-9}
\max_{x \in [0,1]}\int_0^1 \psi (x,y) \di y =\max_{x \in [0,1]} \frac{x^{1-\beta}+ (1-x)^{1-\beta}}{1-\beta} \leq \frac{2^\beta}{1-\beta} =: C_\psi.
\end{align}
Here, we used the fact that the map $x \mapsto \frac{x^{1-\beta}+ (1-x)^{1-\beta}}{1-\beta}$ takes its maximum on $\Omega$ at $x=\frac{1}{2}$. Moreover, this function attains its minimum at the boundary of the domain, i.e., at $x=0$ and $x=1$, 
\begin{align} \label{B-10}
\min_{x \in [0,1]} \int_0^1 \psi (x, y) 
\di y =\int_0^1 \psi (0, y) 
\di y =\int_0^1 \psi (1, y)
\di y =\frac{1}{1-\beta}.
\end{align}
We also notice that in case $\beta = 0$ we clearly have
\begin{align*}
\int_0^1 \psi(x,y) \di y = 1
\end{align*}
and we notice that the estimates in \eqref{B-9}, \eqref{B-10} hold in this case as well. \\
For a function $f:\Omega \rightarrow \bbr$, we can write 
\begin{align} \label{B-11-1}
\int_\Omega \psi(x,y) f(y) \di y = \left( \frac{1}{|\cdot|^\beta} * f \right)(x) =: (\Phi * f)(x),
\end{align}
so we can apply Young's convolution inequality for such terms. It holds 
\begin{align} \label{B-11-2}
\| \Phi \|_1 = \int_\Omega \frac{1}{|x|^\beta} \di x = \frac{1}{1-\beta} =: C_\Phi.
\end{align}
Below, we set 
\begin{equation} \label{B-12}
{\mathcal I}_{a,b}:= \int_{|a - y|\leq |b - y|} (\psi(a, y) - \psi(b, y))  
\di y = \int_{|a - y|\leq |b - y|} \left( \frac{1}{|a-y|^\beta} - \frac{1}{|b-y|^\beta}\right) \di y \\
\end{equation}
and
\begin{align}  \label{B-12-1}
\psi_\varepsilon (x,y) = \frac{1}{|x-y|^\beta +\varepsilon}.
\end{align}

\begin{lemma} \label{L2.5}
The following estimates hold.
\[ \lim\limits_{\rho\to 0}\sup_{\varepsilon>0} \sup_{|v|\leq \rho} \sup_{x\in \Omega} \int_\Omega \left| \frac{1}{|x-v-y|^\alpha +\varepsilon} - \frac{1}{|x-y|^\alpha +\varepsilon} \right| \di y=0, \quad 
\max_{a,b\in [0,1]} {\mathcal I}_{a,b} = \frac{2^{\beta}-1}{1-\beta}.
\]
\end{lemma}
\begin{proof}
\noindent (i)~Note that 
\begin{align*}
&\left| \frac{1}{|x-v-y|^\alpha +\varepsilon} - \frac{1}{|x-y|^\alpha +\varepsilon} \right| = \frac{||x-y|^\alpha - |x-v-y|^\alpha|}{(|x-v-y|^\alpha +\varepsilon)(|x-y|^\alpha +\varepsilon)} \\ &\hspace{0.7cm} \leq \frac{||x-y|^\alpha - |x-v-y|^\alpha|}{|x-v-y|^\alpha |x-y|^\alpha } 
= \left| \frac{1}{|x-v-y|^\alpha} - \frac{1}{|x-y|^\alpha} \right|.
\end{align*}
Therefore it suffices to prove that
\begin{equation} \label{B-13}
\Phi (v) := \sup_{u \in \mathbb{R}}  \int_\Omega \left| \frac{1}{|u-v-y|^\beta}  - \frac{1}{|u-y|^\beta}\right|\di y \rightarrow 0 \quad \mathrm{as} \quad v \rightarrow 0.
\end{equation}
{\it Proof of \eqref{B-13}}:~For $\gamma>0$, we show that $\exists ~v_0 >0$ such that for all $|v|\leq v_0$ 
\[  \Phi (v) < \gamma. \]
We choose $\delta>0$ such that
\[ \frac{8}{1-\beta}  \delta^{1-\beta} < \frac{\gamma}{2} \]
and then we choose $v_0 < \min \left\lbrace \delta, \frac{\gamma \delta^{\beta+1}}{2\beta} \right\rbrace$. Then for every $u\in \mathbb{R}$, one has 
\begin{align} \label{B-2-8}
\begin{split}
& \int_{B_\delta (u) \cap \Omega} \left| \frac{1}{|u-v-y|^\beta}  - \frac{1}{|u-y|^\beta}\right| \di y \\
&\hspace{1cm}\leq \int_{|y-u|<\delta}\frac{1}{|u-v-y|^\beta}\di y + \int_{|y-u|<\delta}\frac{1}{|u-y|^\beta} \di y \\
&\hspace{1cm}\leq  \int_{-\delta-v_0< z-u<\delta+v_0}\frac{1}{|z-u|^\beta} \di z +2 \int_{0<r<\delta}\frac{1}{r^\beta} \di y \\
&\hspace{1cm}\leq  2\int_{0< r<\delta+v_0}\frac{1}{r^\beta} \di r + 2\int_{0<r<\delta}\frac{1}{r^\beta}\di y \\
&\hspace{1cm}\leq \frac{2}{1-\beta} (\delta+v_0)^{1-\beta} + \frac{2}{1-\beta} \delta^{1-\beta} \\
&\hspace{1cm}\leq \frac{4}{1-\beta}  (\delta+v_0)^{1-\beta} \leq \frac{4}{1-\beta}  (2\delta)^{1-\beta} \leq  \frac{8}{1-\beta}  \delta^{1-\beta}<  \frac{\gamma}{2}.
\end{split}
\end{align}
Since 
\begin{align*}
\left| \frac{\di}{\di z} \frac{1}{|z|^\beta} \right| = \frac{\beta}{|z|^{\beta+1}} \quad \mathrm{for} \ z\neq 0,
\end{align*}
$\frac{\di}{\di z} \frac{1}{|z|^\beta}$ is bounded on $\Omega \setminus B_\delta (0)$ by $\frac{\beta}{\delta^{\beta+1}}$.
This yields
\begin{align*}
\int_{\Omega \setminus B_\delta (u)} \left| \frac{1}{|u-v-y|^\beta}  - \frac{1}{|u-y|^\beta}\right| \di y \leq \frac{\beta}{\delta^{\beta+1}} v_0 < \frac{\gamma}{2}.
\end{align*}
We combine the above inequality with \eqref{B-2-8} to derive the desired estimates.  \newline

\noindent (ii)~For the second estimate, we consider two cases:
\[ \mbox{Either}~a = b \quad \mbox{or} \quad a \neq b. \]
\noindent $\bullet$~Case A $(a = b$):~In this case, we use \eqref{B-12} to see
\[ {\mathcal I}_{a,b} =0. \] 
\noindent $\bullet$~Case B $(a \neq b$):~Without loss of generality, we may assume that 
\[ a < b. \]
Note that the domain of integration is equal to 
\begin{align*}
\Big \{  y \in \Omega : |a - y|\leq |b - y| \Big \} = \left[0 , \frac{a+b}{2} \right]
\end{align*} 
and
\begin{align*}
{\mathcal I}_{a,b} = &\int_0^a \left( \frac{1}{|a-y|^\beta} - \frac{1}{|b-y|^\beta}\right) 
\di y  + \int_a^{\frac{a+b}{2}} \left( \frac{1}{|a-y|^\beta} - \frac{1}{|b-y|^\beta}\right) \di y \\
=& \frac{a^{1-\beta} - b^{1-\beta} + (b-a)^{1-\beta}}{1-\beta} + \frac{\left( \frac{b-a}{2} \right)^{1-\beta} - (b-a)^{1-\beta} + \left( \frac{b-a}{2} \right)^{1-\beta}}{1-\beta} \\
=& \frac{a^{1-\beta} - b^{1-\beta} + 2 \left( \frac{b-a}{2} \right)^{1-\beta}}{1-\beta} 
= \frac{a^{1-\beta} - b^{1-\beta} + 2^{\beta} \left( b-a \right)^{1-\beta}}{1-\beta}.
\end{align*}
We set
\[ {\tilde \delta} := b-a, \quad \mbox{i.e.,} \quad b=a + {\tilde \delta}. \]
Then, we have
\begin{align*}
{\mathcal I}_{a,b} = \frac{a^{1-\beta} - (a+ {\tilde \delta})^{1-\beta} + 2^{\beta}  {\tilde \delta}^{1-\beta}}{1-\beta} =: I_{a,\tilde \delta}.
\end{align*}
This implies 
\begin{align*}
\partial_a {I}_{a, {\tilde \delta}} = a^{-\beta} - (a+ {\tilde \delta})^{-\beta} >0.
\end{align*}
For every ${\tilde \delta} >0$ 
\begin{align*}
\max_{a\in [0,1-{\tilde \delta}]} I_{a, {\tilde \delta}}= I_{1-{\tilde \delta}, {\tilde \delta}} = \frac{(1-{\tilde \delta})^{1-\beta} - 1 + 2^\beta {\tilde \delta}^{1-\beta}}{1-\beta}.
\end{align*}
This implies 
\begin{align*}
\max_{a,b \in [0,1]} {\mathcal I}_{a,b}=\max_{\tilde{\delta} \in [0,1]} I_{1-\tilde{\delta},\tilde{\delta}} = I_{0,1}= \frac{2^\beta -1}{1-\beta}.
\end{align*}
This yields the desired second estimate. The case with $b<a$ can be treated similarly.
\end{proof}

\subsection{Properties of Lebesgue functions}\label{sec2.4}
In this subsection, we list several lemmas for $L^p$ functions to be used later. In the sequel, we quote several results by Scorza and Dragoni and Danskin's results from the formulations given in \cite{B-D2022}. 
In the sequel, $\mathscr{S}$ is a metric space.

\begin{lemma}
\emph{(Scorza-Dragoni \cite{B-D2022})} \label{L2.6}
For a Borel set $B \subset \mathbb{R}$, let  $f:\mathbb{R}_+\times B \rightarrow \mathscr{S}$ be a function satisfying the following regularity conditions:
\begin{enumerate}
\item[(i)]
For each $t\ge 0$, 
\[ x\in B \mapsto f(t,x)\in \mathscr{S}~~~\mbox{is $\mathscr{L}$-measurable.} \]
\item[(ii)]
For $\mathscr{L}$-almost all $x\in B$, 
\[ t\in\mathbb{R}_+ \mapsto f(t,x)\in \mathscr{S}~\mbox{is continuous.} \] 
\end{enumerate}
Then, for any $\varepsilon>0$, there exists a compact set $B_{\varepsilon}\subset B$ such that 
\[ |B \backslash B_{\varepsilon} |<\varepsilon \quad \mbox{and} \quad \mbox{the restricted function $f:\mathbb{R}_+\times B_{\varepsilon}\rightarrow\mathscr{S}$ is continuous.} \]
\end{lemma}
Once we have a continuous restriction $f|_{\mathbb{R}_+\times B_{\varepsilon}}$, the following result from Danskin allows us to calculate the differentiation of the map 
$t~\mapsto~\max f|_{\mathbb{R}_+\times B_{\varepsilon}}(t,\cdot) $.
\begin{lemma}
\emph{(Danskin \cite{B-D2022})} \label{L2.7}
For a compact set $K \subset \mathbb{R}$, let $f:\mathbb{R}_+\times K \rightarrow \mathbb{R}$ be a continuous function such that the map $t\in\mathbb{R}_+ \mapsto f(t,x)\in \mathbb{R}$ is differentiable for each $x\in K$. Then, the following assertions hold.
\begin{enumerate}
\item[(i)]
The map $g: t\in\mathbb{R}_+ \mapsto \max_{x\in K}f(t,x)\in \mathbb{R}$ is differentiable $\mathscr{L}$-almost everywhere.
\vspace{0.1cm}
\item[(ii)] 
For $\mathscr{L}$-a.e. $t \in \bbr_+$,
	\begin{equation*}
		\frac{\mathrm{d} g(t) }{\mathrm{d} t} = \max_{x\in \hat{K}(t)} \partial_t f(t,x),
	\end{equation*}
	where $\hat{K}(t):=\text{ \rm argmax}_{x\in K}f(t,x)$.
\end{enumerate}	
\end{lemma}
Now, we are ready to make sure that $\max f|_{\mathbb{R}_+\times K_{\varepsilon}}(t,\cdot)$ from Lemma \ref{L2.6} converges to ${\mathrm{ess\,sup\,}} f(t,\cdot) $ as $\varepsilon$ vanishes, and we can estimate the derivative of ${\mathrm{ess\,sup\,}} f(t,\cdot) $ through the method of approximation.
\begin{lemma} 
\emph{\cite{MWX2025}} \label{L2.8}
	For a bounded Borel set $B \subset\mathbb{R}^d$ and function $g\in L^{p}(B;\mathbb{R})$, let $\{B_\varepsilon\}_{\varepsilon>0}$ be compact subsets of $B$ such that 
	\[ |B_\varepsilon^c|<\varepsilon \quad \mbox{and} \quad  g|_{B_{\varepsilon}}~\mbox{is continuous}. \]
	Then, we have
	\begin{equation*}
		\lim_{\varepsilon\rightarrow 0} \max_{x\in B_\varepsilon} g(x)= \underset{x\in B}{\mathrm{ess\,sup\,}} g(x).
	\end{equation*}
\end{lemma}
\begin{lemma} 
\emph{\cite{MWX2025}}
	\label{L2.9}
	Let $H$ and $\{H_n\}_{n\in\mathbb{N}}$ be measurable, almost everywhere differentiable mappings  from $\mathbb{R}_+$ to $\mathbb{R}$. Suppose that the following conditions hold.
\begin{enumerate}	
\item[(i)]
$H_n$ converges to $H$ pointwise.
\item[(ii)]
$H_n$ and  $\frac{\mathrm{d}}{\mathrm{d} t} H_n$ are locally Lipschitz continuous with Lipschitz constants independent of $n$. 
\end{enumerate}
Then, for almost every $t\in\mathbb{R}_+$, 
	\begin{equation*}
		\liminf_{n\rightarrow\infty} \frac{\mathrm{d}}{\mathrm{d} t} H_{n}(t)
		\le
		\frac{\mathrm{d}}{\mathrm{d} t} H(t) \le \limsup_{n\rightarrow\infty} \frac{\mathrm{d}}{\mathrm{d} t} H_{n}(t).
	\end{equation*}
\end{lemma}

\begin{lemma} \label{L2.13}
Let $f\in L^p (D)$ for $p \in [1,\infty)$ and $D\subset \bbr^d$. Then it holds
\begin{align*}
\lim_{v\rightarrow 0}\int_D |f(x-v) - f(x)|^p \di x = 0.
\end{align*} 
\end{lemma}
\begin{proof}
It holds
\begin{align} \label{B-2-9}
\int_D |f(x-v) - f(x)|^p \di x \leq \|\tilde{f}(\cdot-v) - \tilde{f}(\cdot)\|^p_{L^p(\bbr^d)} ,
\end{align}
where $\tilde{f}$ denotes the zero extension of $f$ to $\bbr^d$. Since $C_c^\infty (\bbr^d)$ is dense in $L^p (\bbr^d)$, for every $\varepsilon>0$ there exists $g\in C_c^\infty (\bbr^d)$ such that
\begin{align} \label{B-2-10}
\| \tilde{f} - g\|_{L^p(\bbr^d)}^p <\frac{\varepsilon}{3^p}.
\end{align}
By the triangle inequality, it holds
\begin{align} \label{B-2-11}
\begin{aligned}
&\|\tilde{f}(\cdot-v) - \tilde{f}(\cdot)\|_{L^p(\bbr^d)}^p \leq 3^{p-1} \bigg( \|\tilde{f}(\cdot-v) - g(\cdot-v)\|_{L^p(\bbr^d)}^p \\
&\hspace{1cm}+ \|g(\cdot-v) - g(\cdot)\|_{L^p(\bbr^d)}^p + \|g - \tilde{f}\|_{L^p(\bbr^d)}^p \bigg).
\end{aligned}
\end{align}
Since translations preserve the $L^p$-norm on $\bbr^d$, it holds
\begin{align} \label{B-2-12}
\|\tilde{f}(\cdot-v) - g(\cdot-v)\|_{L^p(\bbr^d)}^p = \| \tilde{f} - g\|_{L^p(\bbr^d)}^p  <\frac{\varepsilon}{3^p},
\end{align}
and since $g$ is continuous, there exists $v_0>0$ such that for all $|v|<v_0$ it holds
\begin{align} \label{B-2-13}
\|g(\cdot-v) - g(\cdot)\|_{L^p(\bbr^d)}^p < \frac{\varepsilon}{3^p}.
\end{align}
Combining \eqref{B-2-10}, \eqref{B-2-12} and \eqref{B-2-13}, we use \eqref{B-2-11} and \eqref{B-2-9} to show the claim. 
\end{proof}

\section{The singular continuum Kuramoto model} \label{sec:3}
\setcounter{equation}{0}
In this section, we study a global well-posedness and emergent behaviors of the SCKM \eqref{A-2} on the following two domains for $\Omega$:
\[ \mbox{Either}~\Omega = \bbr^d \quad \mbox{or} \quad \Omega = [0, 1]. \]
\subsection{From the SCKM to the linear equation} \label{sec:3.1}
In this subsection, we consider the special case of \eqref{A-2} with following system parameters:
\[ \alpha = 0, \quad  \psi (x,y) = \delta_x (y) \quad \mbox{and} \quad \Omega=\mathbb{R}^d. \]
In this setting, the Cauchy problem \eqref{A-2} becomes 
\begin{equation} \label{C-0}
\begin{cases}
\displaystyle \partial_t \theta(t, x) = \nu(x) + \kappa \int_{\bbr^d} \sin( \theta(t, y) - \theta(t, x)) \di \delta_x (y) = \nu(x), \quad t > 0,~~x \in \bbr^d, \\
\displaystyle \theta(0, x) = \theta^{\mathrm{in}}(x),  \quad  x \in \bbr^d.
\end{cases}
\end{equation}
Let $\eta$ be a standard mollifier on $\bbr^d$ such that 
\[ \eta \geq 0, \quad \int_{\bbr^d} \eta(x) \di x = 1, \quad \mbox{supp}~(\eta) \subset B_1(0), \]
where $B_1(0)$ is the unit ball with a center $0$ and a radius $1$. We also let $\eta_\varepsilon$ be the mollifier defined by 
\[ \eta_\varepsilon(\cdot)=\varepsilon^{-d} \eta(\cdot/\varepsilon). \]
It satisfies 
\[ \eta_{\varepsilon} \geq 0, \quad \int_{\bbr^d} \eta_{\varepsilon}(x) \di x = 1, \quad  \mbox{supp}~\eta _{\varepsilon} \subset B_{\varepsilon}(0). \]
Then a sequence $(\eta_\varepsilon)_{\varepsilon >0}$ approximates $\delta$ as $\varepsilon \rightarrow 0$. In the following, we use a handy notation $\eta_\varepsilon (x,y) := \eta_\varepsilon (x-y)$, for which it holds
\begin{align*}
\eta_\varepsilon (x,\cdot) \rightarrow \delta_x (\cdot) \quad \mathrm{as} \ \varepsilon \rightarrow 0
\end{align*} 
in sense of distributions on $\bbr^d$ for each fixed $x\in \bbr^d$. Consider the following regularized equation:
\begin{equation} \label{C-1}
\begin{cases}
		\partial_t \theta_\varepsilon(t, x) = \nu(x) + \kappa \displaystyle\int_{\bbr^d} \eta_\varepsilon (x,y) \sin\big( \theta_\varepsilon(t, y) - \theta_\varepsilon (t, x) \big)\di y, \quad  t > 0, \; x \in \bbr^d, \\
		\theta_\varepsilon (0, x) = \theta^{\mathrm{in}}(x), \quad x \in \bbr^d.
	\end{cases}
\end{equation} 
The Cauchy problem \eqref{C-0} admits a unique global solution:
\begin{align*}
\theta(t,x) = \theta^\mathrm{in} (x) + \nu(x) t.
\end{align*}
Since the integro-differential equation \eqref{C-1} contains an integral with respect to the spatial variable, we are able to define
\begin{align*}
f_\varepsilon : L^p(\Omega) \rightarrow L^p(\Omega), \ q \mapsto f_\varepsilon(q)(x) := \nu(x)+ \kappa \displaystyle\int_\Omega \eta_\varepsilon (x,y) \sin (q(y) - q(x)) \di y.
\end{align*}
It can be easily shown that $f_\varepsilon$ is Lipschitz continuous, since $\eta_\varepsilon$ is bounded and $\sin$ is Lipschitz continuous. Then we can write \eqref{C-1} as 
\begin{align*}
\begin{cases}
\partial_t \theta_\varepsilon (t,x) = f_\varepsilon(\theta_\varepsilon(t,x)), \\
\theta_\varepsilon(0,x) = \theta^\mathrm{in}(x).
\end{cases}
\end{align*} 
Therefore, the regularized Cauchy problem \eqref{C-1} has a unique global solution $\theta_\varepsilon \in C([0,T]; L^p(\Omega))$ due to Banach fixed point theorem (or equivalently, Picard-Lindelöf's theorem in Banach spaces). In the following proposition, we show that $\theta_\varepsilon (t)$ converges to $\theta (t)$ in $L^p$ sense.
\begin{proposition}\label{P3.1}
Suppose that natural frequency function and initial data satisfy 
\begin{align*}
T \in (0, \infty],\quad \nu \in L^p (\mathbb{R}^d), \quad \theta^{\mathrm{in}} \in L^p(\mathbb{R}^d), \quad ~\mbox{for $p \in [1, \infty)$},
\end{align*}
and let $\theta, \theta_\varepsilon$ be global solutions to \eqref{C-0} and \eqref{C-1} in the time interval $[0, T)$, respectively. Then, $\theta_\varepsilon$ converges to $\theta$ in ${C}([0,T); L^p(\bbr^d))$ as $\varepsilon \to 0$. 
\end{proposition}
\begin{proof} We split the proof into two steps. \newline

\noindent $\bullet$~Step A: We claim that 
\begin{align}
\begin{aligned} \label{C-1-1}
& \| \theta_\varepsilon (t) -  \theta(t ) \|_p^p \\
& \hspace{1cm} \leq  e^{\left( 2\times 3^{p-1} (\kappa t)^p \right)} (3t)^{p-1} \kappa^p \frac{t^{p+1}}{p+1} \int_{\bbr^d} \left( \int_{\bbr^d} \eta_\varepsilon (x,y) | \nu (y) - \nu (x)| \di y \right)^p \di x \\
&\hspace{1.2cm} + e^{\left( 2\times 3^{p-1} (\kappa t)^p \right)} 6^{p-1} (\kappa t)^p \int_{\bbr^d} \left( \int_{\bbr^d} \eta_\varepsilon (x,y)|\theta^\mathrm{in}(x) - \theta^\mathrm{in}(y)| \di y \right)^p \di x.
\end{aligned}
\end{align}
{\it Proof of \eqref{C-1-1}}:~It follows from \eqref{C-0} and \eqref{C-1} that  for $(t, x) \in \bbr_+ \times \bbr^d$, 
\begin{align}
\begin{aligned} \label{C-2}
&|\theta_\varepsilon (t,x) -  \theta (t,x)|^p \\
& \hspace{0.5cm} \leq  \left( \kappa \int_0^t \int_{\bbr^d} \eta_\varepsilon (x,y) |\sin (\theta_\varepsilon (\tau,y) - \theta_\varepsilon (\tau,x) )| \di y \di \tau \right)^p \\
&  \hspace{0.5cm} \leq  \kappa^p t^{p-1} \int_0^t \left( \int_{\bbr^d} \eta_\varepsilon (x,y) |\theta_\varepsilon (\tau,y) - \theta_\varepsilon (\tau,x) | \di y \right)^p \di \tau\\
&   \hspace{0.5cm} \leq \kappa^p t^{p-1} \int_0^t \Big( \int_{\bbr^d} \eta_\varepsilon (x,y) ( |\theta_\varepsilon (\tau,y) - \theta (\tau,y) | + |\theta (\tau,y) - \theta (\tau,x) |  \\
&  \hspace{2.5cm} + |\theta_\varepsilon (\tau,x) - \theta (\tau,x) | ) \di y \Big)^p \di \tau,
\end{aligned}
\end{align}
where we used H\"older's inequality in the second step. \newline

Next, we integrate \eqref{C-2} over $x\in \bbr^d$ to find
\begin{align} 
&\| \theta_\varepsilon (t) -  \theta(t ) \|_p^p \leq  (3t)^{p-1} \kappa^p  \int_0^t \int_{\bbr^d} \left( \int_{\bbr^d} \eta_\varepsilon (x,y) |\theta_\varepsilon (\tau,x) - \theta (\tau,x) |\di y \right)^p \di x \di \tau \nonumber\\ 
&\hspace{0.7cm}+ (3t)^{p-1} \kappa^p \int_0^t \int_{\bbr^d} \left(\int_{\bbr^d} \eta_\varepsilon (x,y) |\theta (\tau,y) - \theta (\tau,x) |\di y \right)^p \di x \di \tau \nonumber\\
&\hspace{0.7cm}+ (3t)^{p-1} \kappa^p \int_0^t \int_{\bbr^d} \left(\int_{\bbr^d} \eta_\varepsilon (x,y) |\theta_\varepsilon (\tau,y) - \theta (\tau,y) |\di y \right)^p \di x \di \tau\nonumber \\
&\hspace{0.5cm} \leq (3t)^{p-1} \kappa^p \int_0^t \int_{\bbr^d} |\theta_\varepsilon (\tau,x) - \theta (\tau,x) |^p \di x \di \tau \nonumber\\ 
&\hspace{0.7cm}+ (3t)^{p-1} \kappa^p \int_0^t \int_{\bbr^d} \left( \int_{\bbr^d} \eta_\varepsilon (x,y) |\theta (\tau,y) - \theta (\tau,x) | \di y \right)^p \di x \di \tau \nonumber\\
&\hspace{0.7cm} + (3t)^{p-1} \kappa^p \int_0^t \int_{\bbr^d} \left( \eta_\varepsilon * |\theta_\varepsilon (\tau,\cdot) - \theta (\tau,\cdot) | (x) \right)^p \di x \di \tau \nonumber \\
&\hspace{0.5cm} \leq  (3t)^{p-1} \kappa^p \int_0^t \|\theta_\varepsilon (\tau) - \theta (\tau) \|_p^p  \di \tau \nonumber \\ 
&\hspace{0.7cm} + (3t)^{p-1} \kappa^p \int_0^t \int_{\bbr^d} \left( \int_{\bbr^d} \eta_\varepsilon (x,y) |\theta (\tau,y) - \theta (\tau,x) |\di y \right)^p \di x \di \tau \nonumber \\ \label{C-3}
&\hspace{0.7cm}+ (3t)^{p-1} \kappa^p \int_0^t \|\eta_\varepsilon \|_{1}^p \| \theta_\varepsilon (\tau) - \theta (\tau) \|_p^p \di \tau,
\end{align}
where we used Young's convolution inequality in the last step. Then, we apply Gr\"onwall's lemma for \eqref{C-3} to find the desired claim:
\begin{align} 
\begin{aligned} \label{C-3-10}
& \| \theta_\varepsilon (t) -  \theta(t ) \|_p^p \\
& \hspace{0.5cm} \leq e^{\left( 2\times 3^{p-1} (\kappa t)^p \right)} (3t)^{p-1} \kappa^p \int_0^t \int_{\bbr^d} \left( \int_{\bbr^d} \eta_\varepsilon (x,y) |\theta (\tau,y) - \theta (\tau,x) | \di y \right)^p \di x \di \tau \\
&\hspace{0.5cm} \leq e^{\left( 2\times 3^{p-1} (\kappa t)^p \right)} (3t)^{p-1} \kappa^p \int_0^t \int_{\bbr^d} \left( \int_{\bbr^d} \eta_\varepsilon (x,y) (|\theta^\mathrm{in}(x) - \theta^\mathrm{in}(y)|+| \nu (y) - \nu (x)| \tau )  \di y \right)^p \di x \di \tau \\
& \hspace{0.5cm} \leq e^{\left( 2\times 3^{p-1} (\kappa t)^p \right)} (6t)^{p-1} \kappa^p \frac{t^{p+1}}{p+1} \int_{\bbr^d} \left( \int_{\bbr^d} \eta_\varepsilon (x,y) | \nu (y) - \nu (x)| \di y \right)^p \di x\\
& \hspace{0.7cm} + e^{\left( 2\times 3^{p-1} (\kappa t)^p \right)} 6^{p-1} (\kappa t)^p \int_{\bbr^d} \left( \int_{\bbr^d} \eta_\varepsilon (x,y)|\theta^\mathrm{in}(x) - \theta^\mathrm{in}(y)| \di y \right)^p \di x.
\end{aligned}
\end{align}
\noindent $\bullet$~Step B:~Since $\eta_\varepsilon (z) dz$ is a probability measure on $\bbr^d$, we can use Jensen's inequality as follows:
\begin{align*}
\begin{aligned}
&\int_{\bbr^d} \left( \int_{\bbr^d} \eta_\varepsilon (x,y) | \nu (y) - \nu (x)| \di y \right)^p \di x \leq  \int_{\bbr^d}  \int_{\bbr^d} \eta_\varepsilon (x,y) | \nu (y) - \nu (x)|^p \di y  \di x \\
&\hspace{2cm} \leq  \int_{\bbr^d}  \eta_\varepsilon (z) \int_{\bbr^d}   | \nu (x-z) - \nu (x)|^p  \di x \di z.
\end{aligned}
\end{align*}
Next, we use Lemma \ref{L2.13} and the fact that $\eta_\varepsilon(z) \di z$ is a probability measure on $\bbr^d$ and is supported on $B_\varepsilon (z)$, as well as Lebesgue dominant convergence theorem to obtain
\begin{align}
\begin{aligned} \label{C-3-11}
\lim_{\varepsilon \to 0 } \int_{\bbr^d}  \eta_\varepsilon (z) \int_{\bbr^d}   | \nu (x-z) - \nu (x)|^p  \di x \di z = 0.
\end{aligned}
\end{align}
The same argument holds for the second term:
\begin{align} \label{C-3-12}
\lim_{\varepsilon \rightarrow 0} \int_{\bbr^d}  \left( \int_{\bbr^d} \eta_\varepsilon (x,y) |\theta^\mathrm{in} (y) - \theta^\mathrm{in} (x) |  \di y \right)^p \di x = 0.
\end{align}
Finally, we combine \eqref{C-1-1}, \eqref{C-3-11} and \eqref{C-3-12} to find the desired estimate.
\end{proof}

\subsection{A global well-posedness}  \label{sec:3.2}
In this subsection, we provide existence of a global solution to \eqref{A-2} on $\Omega = [0,1]$. For this, we first recall the following theorem a.k.a. Kolmogorov-Riesz-Fr\'echet theorem.
\begin{proposition} 
\emph{(Theorem 4.25, \cite{Brezis2010})} \label{P3.2}
	Let $\mathcal{F}$ be a bounded set in $L^p (\bbr^d)$ with $1\leq p<\infty$ such that 
	\begin{align*}
		\lim_{|v|\rightarrow 0} \| \tau_v f - f \|_p = 0 \ \mathrm{uniformly \ in} \ f\in \mathcal{F}.
	\end{align*}
	Then the closure of $\mathcal{F} \Big|_{\Omega}$ in $L^p (\Omega)$ is compact for any measurable set $\Omega \subset \bbr^d$ with finite measure. Here, $\mathcal{F} \Big|_{\Omega}$ denotes the restrictions to $\Omega$ of the functions in $\mathcal{F}$.
\end{proposition}
Now, we are ready to provide the first main result on the global unique solvability of \eqref{A-2}.  
\begin{theorem} \label{T3.1}
Suppose that $T$, domain, system parameters and initial data satisfy 
\[ T \in (0, \infty],  \quad \Omega = [0, 1], \quad  0 < \alpha, \beta < 1, \quad  \nu \in L^2(\Omega), \quad \theta^{\mathrm{in}} \in L^2(\Omega). \]
Then, there exists at least one solution $\theta \in {C}((0,T); L^2(\Omega))$ to \eqref{A-2}.
\end{theorem}

\begin{proof}
Consider the regularized Cauchy problem corresponding to \eqref{A-2}:
\begin{equation} \label{C-4}
	\begin{cases}
		\partial_t \theta_\varepsilon (t, x) = \nu(x) + \kappa \displaystyle\int_{\Omega} \psi_\varepsilon(x, y) h_\varepsilon \big( \theta_\varepsilon(t, y) - \theta_\varepsilon(t, x) \big) \di y, & t > 0, \; x \in \Omega, \\
		\theta_\varepsilon(0, x) = \theta^{\mathrm{in}}(x), & x \in \Omega,
	\end{cases}
\end{equation}
where $\psi_\varepsilon$ is defined as in \eqref{B-12-1} and $h_\varepsilon$ is defined as in \eqref{B-3-1}. We can define 
\begin{align*}
g_\varepsilon : L^2 (\Omega) \rightarrow L^2 (\Omega), \ q \mapsto g_\varepsilon(q)(x) := \nu(x)+ \kappa \displaystyle\int_\Omega \psi_\varepsilon (x,y) h_\varepsilon (q(y) - q(x)) \di y.
\end{align*}
It can be easily shown that $g_\varepsilon$ is Lipschitz continuous, since $\psi_\varepsilon$ is bounded and $h_\varepsilon$ is Lipschitz continuous. Then we can write \eqref{C-4} as 
\begin{align*}
\begin{cases}
\partial_t \theta_\varepsilon (t,x) = g_\varepsilon(\theta_\varepsilon(t,x)), \\
\theta_\varepsilon(0,x) = \theta^\mathrm{in}(x).
\end{cases}
\end{align*} 
Therefore, the regularized Cauchy problem \eqref{C-4} has a unique global solution $\theta_\varepsilon \in C([0,T]; L^2(\Omega))$ due to Picard-Lindelöf's Theorem. Then, we can establish the existence of a solution to \eqref{A-2} using Proposition \ref{P3.2}. Since the proof is very lengthy, we briefly sketch a proof strategy as follows. 
\vspace{0.1cm}
\begin{itemize}
\item Step A:~We show that the sequence $(\theta_\varepsilon)_{\varepsilon>0}$ is equicontinuous in $t$ and relatively compact for every fixed $t\in [0,T]$, which allows us to extract a convergent subsequence in the space $C([0,T);L^2(\Omega))$.
\vspace{0.1cm}
\item Step B:~We show that the limit of this subsequence solves \eqref{A-2}. 
\end{itemize}
The detailed argument for the above steps will be given in Appendix \ref{App-A}.
\end{proof}
\begin{remark} \label{R3.1}
It is clear that if $\theta_\varepsilon$ solves \eqref{C-4} with initial data $\theta^\mathrm{in}$, then $\theta_\varepsilon + 2\pi$ also solves \eqref{C-4} with initial data $\theta^\mathrm{in} + 2\pi$. The same is true for $\theta$ solving \eqref{A-2}. Therefore, throughout this paper, we can consider all angles modulo $2\pi$, i.e. $\theta_\varepsilon(t,x), \theta(t,x) \in (-\pi,\pi]$. We will establish Filippov uniqueness for $S^1$-valued angles $\theta$ in Theorem \ref{T4.1} later. Since classical solutions are also solutions in the weaker Filippov sense, we automatically obtain uniqueness of classical solutions with this argument.
\end{remark}
\subsection{Identical oscillators}\label{sec:3.3}
In this subsection, we study the emergent behaviors of  a global solution $\theta(t,x)$ to the SCKM \eqref{A-2} with $\nu \equiv 0$ on a bounded domain $\Omega = [0, 1].$ \newline

For $n \in {\mathbb N}$ there exists a compact domain $\Omega_n$ such that the restriction $\theta(t,\cdot) \vert_{\Omega_n}$ is continuous. We set 
 \[
 \mathcal{D}_{n}(\theta(t)) = \max_{x,y\in \Omega_{n}} |\theta(t,x) - \theta(t,y)|= \max_{x\in \Omega_{n}} \theta(t,x) - \min_{y\in \Omega_{n}} \theta(t,y), \quad \mathcal{D}(\theta(t)) = \sup_{x,y\in \Omega} |\theta(t,x) - \theta(t,y)|.
 \]
Then, $\mathcal{D}_{n}(\theta(t))$ is Lipschitz continuous, and it is differentiable almost everywhere. First, we provide the following lemma.
\begin{lemma}\label{L3.1}
For $T \in (0, \infty)$, let $\theta = \theta(t,x)$ be a solution to \eqref{A-2} that satisfies the following a priori condition:
\[
	\sup\limits_{0\le t\le T} \sup\limits_{x\in\Omega} \Big( \nu(x) + \kappa \displaystyle\int_{\Omega} \psi(x,y) h( \theta(t, y) - \theta(t, x))\di y \Big) \le C.\]
Then, the following assertions hold.
\begin{enumerate}
\item[(i)]		
The diameter function $\mathcal{D}(\theta)$ is Lipschitz continuous in $[0,T)$. 
\vspace{0.1cm}
\item[(ii)]	
There exists a bounded index sequence $(\Omega_{n})_{n\in \mathbb{R}}$ with $\Omega_n \subset \Omega$ and $|\Omega \setminus \Omega_n | < \frac{1}{n}$ such that:~ for a.e. $t>0$,
\begin{equation} 
\begin{cases} \label{New-1}
\displaystyle \mathcal{D}(\theta(t))=\lim\limits_{n\to \infty}	\mathcal{D}_{n}(\theta(t)). \\
\displaystyle  \frac{\di}{\di t} \mathcal{D}(\theta(t))\le \limsup\limits_{n\to \infty} \frac{\di}{\di t}\mathcal{D}_{n}(\theta(t))=\limsup\limits_{n\to \infty} \left( \frac{\di}{\di t}\max_{x\in \Omega_{n}} \theta(t,x) -\frac{\di}{\di t} \min_{y\in \Omega_{n}} \theta(t,y) \right).
\end{cases}		
\end{equation}
\end{enumerate}		
\end{lemma}
\begin{proof}
We combine Lemma \ref{L2.8} and Lemma \ref{L2.9} to derive the desired estimates.
\end{proof}
\begin{remark}\label{R3.2}
We set 
\begin{equation} \label{New-2}
M(t) =\left\{x_M~\Big| ~\theta(t,x_M)=\sup_{x\in \Omega} \theta(t,x)\right\}, \quad  m(t) =\left\{x_m~\Big| ~\theta(t,x_m)=\inf_{x\in \Omega} \theta(t,x)\right\}.
\end{equation}
Then, the result of Lemma \ref{L3.1} yields the existence of some time-varying indices $x_M\in M(t)$ and $x_m\in m(t)$ such that 
\begin{align}\label{New-3}
		\frac{\di}{\di t} \mathcal{D}(\theta(t))=\frac{\di}{\di t}\sup_{x\in \Omega} \theta(t,x) -\frac{\di}{\di t} \inf_{y\in \Omega} \theta(t,y)=\frac{\di}{\di t}\theta(t,x_M)-\frac{\di}{\di t}\theta(t,x_m), \quad \text{a.e.} \quad t>0.
\end{align}
Otherwise, we can choose some approximation sequences $(\Omega_n, x_{M_n(t) }, x_{m_n(t)})$ and use \eqref{New-1} to derive the estimate of the derivative of  $\mathcal{D}(\theta(t))$.
\end{remark}
We are now ready to provide complete synchronization result. Throughout the paper, as long as there is no confusion, we suppress $t$-dependence in $\theta$ and use a handy notation:
 \[ \theta(x) \equiv \theta(t,x), \quad {\mathcal D}_0 := {\mathcal D}(\theta^{\mathrm{in}}). \]
\begin{theorem} \label{T3.2}
Suppose that natural frequency function and initial data satisfy 
\[ \nu \equiv 0, \quad \theta^{\mathrm{in}} \in L^2(\Omega), \quad 0\leq \mathcal{D}_0 <\pi, \]
and let $\theta = \theta(t,x)$ be a global solution to \eqref{A-2} with $\alpha \in (0,1)$. Then, the following assertions hold. 
\begin{enumerate}
\item[(i)]
The diameter functional ${\mathcal D}(\theta)$ is contractive in the sense that 
\[ \sup\limits_{0\le t <\infty } 	\mathcal{D}(\theta(t))\le \mathcal{D}_0. \]
\item[(ii)]
Finite time complete phase synchronization occurs, i.e., there exists $t_c > \frac{(1-\beta)\mathcal{D}_0^{1+\alpha} }{\alpha \kappa \sin(\mathcal{D}_0) (2-2^\beta)} $   such that 
\[ {\mathcal D}(\theta(t)) = 0, \quad t \geq t_c. \]
\end{enumerate}
\end{theorem}
\begin{proof}
\noindent (i)~First, we use \eqref{B-9} and $h( \theta(y) - \theta(x))\le 1$ to see 
\[	\sup\limits_{0\le t<\infty}	\sup\limits_{x\in\Omega}  \Big( \kappa \int_{\Omega} \psi(x,y) h( \theta(y) - \theta(x))\di y \Big) \le  C_{\psi}\kappa.\]
By Lemma \ref{L3.1} and Remark \ref{R3.2}, we obtain 
\[\frac{\di}{\di t} \mathcal{D}(\theta) =  \frac{\di  \theta(x_M)}{\di t}  - \frac{\di  \theta(x_m)}{\di t},\]
where $x_M\in M(t)$ and $x_m\in m(t)$ are defined in \eqref{New-2}. \newline

\noindent $\bullet$~Case A: Suppose that 
\[  \mathcal{D}_0 \neq 0. \]
For every $\tilde t$ with $0<\mathcal{D} (\theta(\tilde t)) <\pi$, we have 
	\begin{align*}
	\frac{\di}{\di t} \mathcal{D}(\theta(t)) \Big\vert_{t= \tilde t} = \kappa \displaystyle\int_{\Omega} \psi(x_M, y) \frac{\sin\big( \theta(y) - \theta(x_M) \big)}{|\theta(y) - \theta(x_M)|_o^{\alpha}}\di y - \kappa \displaystyle\int_{\Omega} \psi(x_m, y) \frac{\sin\big( \theta(y) - \theta(x_m) \big)}{|\theta(y) - \theta(x_m)|^{\alpha}_o} \di y <0,
\end{align*}
since it holds
\begin{align*}
\sin(\theta(y)-\theta(x_M))<0 \quad \mathrm{and} \quad \sin(\theta(y)-\theta(x_m))>0
\end{align*}
due to the definition of $x_M$ and $x_m$. This implies that there exists a sufficiently small $\varepsilon$ such that $\mathcal{D}(\theta(t))$ is strictly decreasing for $t\in [0,\varepsilon)$. \newline

\noindent Now, we claim that 
\begin{equation} \label{C-5}
	\mathcal{D}(\theta(t)) <\mathcal{D}_0 \quad \mathrm{for} \ t>0.
\end{equation}
Suppose the contrary holds, i.e., there exists finite time $t_1>0$ such that
\begin{align*}
	\mathcal{D}(\theta(t)) < \mathcal{D}_0 \quad \mathrm{for} \quad t\in (0,t_1) \quad \mathrm{and} \quad \mathcal{D}(\theta(t_1)) = \mathcal{D}_0. 
\end{align*}
Recall that
\begin{equation} \label{C-6}
	\frac{\di}{\di t} \mathcal{D}(\theta (t)) \Big\vert_{t=t_1} < 0.
\end{equation}
Once the relation \eqref{C-6} holds, then it is contradictory to 
\[ \mathcal{D}(\theta(t)) < \mathcal{D}(\theta(t_1)) \quad \mbox{for $t<t_1$}. \]
\noindent $\bullet$~Case B: Suppose that 
\[\mathcal{D}_0=0. \]
Then, we have 
\[ \mathcal{D}(\theta(t)) \equiv 0, \quad t \geq 0. \]
Finally, we combine the results of Case A and Case B to find the desired estimate.

\vspace{0.2cm}

\noindent (ii)~We claim that ${\mathcal D}(\theta)$ satisfies 
\begin{equation}\label{C-7}
		\frac{\di}{\di t} \mathcal{D}(\theta(t)) \leq -\kappa\frac{\sin\big( \mathcal{D}_0 \big)}{\mathcal{D}_0}  \frac{2-2^{\beta}}{1-\beta}  \mathcal{D}(\theta(t))^{1-\alpha}, \quad \mbox{a.e.}~t > 0.
\end{equation}
{\it Derivation of \eqref{C-7}}:~We make use of Lemma \ref{L2.4} to find
	\begin{align*}
	\begin{aligned}
	& \frac{\sin\big( \theta(y) - \theta(x_m) \big)}{|\theta(y) - \theta(x_m)|^{\alpha}_o} \geq \frac{\sin \mathcal{D}(\theta (t)) }{\mathcal{D}(\theta (t))} \big( \theta(y) - \theta(x_m) \big)^{1-\alpha} \geq 
	\frac{\sin \mathcal{D}_0 }{\mathcal{D}_0} \big( \theta(y) - \theta(x_m) \big)^{1-\alpha}, \\
	& \frac{\sin\big( \theta(x_M) - \theta(y) \big)}{|\theta(x_M) - \theta(y)|^{\alpha}_o} \geq \frac{\sin \mathcal{D}(\theta (t)) }{\mathcal{D}(\theta (t))} \big( \theta(x_M) - \theta(y) \big)^{1-\alpha} \geq 
	\frac{\sin \mathcal{D}_0 }{\mathcal{D}_0} \big( \theta(x_M) - \theta(y) \big)^{1-\alpha}.
	\end{aligned}
	\end{align*}
	These yield
	\[
		\frac{\sin\big( \theta(y) - \theta(x_M) \big)}{|\theta(y) - \theta(x_M)|^{\alpha}_o} = -\frac{\sin\big( \theta(x_M) - \theta(y) \big)}{|\theta(x_M) - \theta(y)|^{\alpha}_o} \leq - \frac{\sin \mathcal{D}_0 }{\mathcal{D}_0} \big( \theta(x_M) - \theta(y) \big)^{1-\alpha},
	\]
and 
	\begin{align*}
	\begin{aligned}
		&\frac{\di}{\di t} \mathcal{D}(\theta(t))  = \kappa \displaystyle\int_{\Omega} \psi(x_M, y) \frac{\sin\big( \theta(y) - \theta(x_M) \big)}{|\theta(y) - \theta(x_M)|_o^{\alpha}}\di y  - \kappa \displaystyle\int_{\Omega} \psi(x_m, y) \frac{\sin\big( \theta(y) - \theta(x_m) \big)}{|\theta(y) - \theta(x_m)|^{\alpha}_o} \di y \\
		&\hspace{0.5cm}  \leq  -\kappa \displaystyle\int_{\Omega} \psi(x_M, y) \frac{\sin \mathcal{D}_0 }{\mathcal{D}_0} \big( \theta(x_M) - \theta(y) \big)^{1-\alpha} \di y - \kappa \displaystyle\int_{\Omega} \psi(x_m, y) \frac{\sin \mathcal{D}_0 }{\mathcal{D}_0} \big( \theta(y) - \theta(x_m) \big)^{1-\alpha} \di y \\
		&\hspace{0.5cm} \leq  -\kappa \frac{\sin \mathcal{D}_0 }{\mathcal{D}_0} \int_{\Omega} (\psi(x_M, y) - \psi(x_m, y))  \big( \theta(x_M) - \theta(y) \big)^{1-\alpha} \di y \\
		&\hspace{0.7cm}  - \kappa \frac{\sin \mathcal{D}_0 }{\mathcal{D}_0} \int_{\Omega} \psi(x_m, y) \left(\big( \theta(y) - \theta(x_m) \big)^{1-\alpha}+ \big( \theta(x_M) - \theta(y) \big)^{1-\alpha} \right) \di y.
	\end{aligned}		
	\end{align*}
			Then, we use Lemma \ref{L2.2} to see
			\[\big( \theta(y) - \theta(x_m) \big)^{1-\alpha}+ \big( \theta(x_M) - \theta(y) \big)^{1-\alpha} \ge \left( \theta(x_M) - \theta(x_m) \right)^{1-\alpha}.\]
Therefore, we have 
			\begin{align} 
			\begin{aligned} \label{New7}
		 \frac{\di}{\di t} \mathcal{D}(\theta(t)) &\leq -\kappa \frac{\sin \mathcal{D}_0 }{\mathcal{D}_0} \int_{\Omega} (\psi(x_M, y) - \psi(x_m, y))  \big( \theta(x_M) - \theta(y) \big)^{1-\alpha} \di y \nonumber\\
		&\hspace{0.2cm}- \kappa \frac{\sin \mathcal{D}_0 }{\mathcal{D}_0}  \mathcal{D}(\theta(t))^{1-\alpha} \int_\Omega \psi (x_m,y ) \di y\nonumber \\
		&= \kappa \frac{\sin \mathcal{D}_0 }{\mathcal{D}_0} \int_{|x_M - y|\geq |x_m - y|} (\psi(x_m, y) - \psi(x_M, y))  \big( \theta(x_M) - \theta(y) \big)^{1-\alpha} \di y \nonumber\\
		&\hspace{0.2cm}+ \kappa \frac{\sin \mathcal{D}_0 }{\mathcal{D}_0} \int_{|x_M - y|< |x_m - y|} (\psi(x_m, y) - \psi(x_M, y))  \big( \theta(x_M) - \theta(y) \big)^{1-\alpha} \di y \nonumber\\
		&\hspace{0.2cm}- \kappa \frac{\sin \mathcal{D}_0 }{\mathcal{D}_0}  \mathcal{D}(\theta(t))^{1-\alpha} \int_\Omega \psi (x_m,y )\di y \nonumber\\
&  \leq  \kappa \frac{\sin \mathcal{D}_0 }{\mathcal{D}_0} \int_{|x_M - y|\geq |x_m - y|} (\psi(x_m, y) - \psi(x_M, y))  \big( \theta(x_M) - \theta(y) \big)^{1-\alpha}  \di y \nonumber\\
		&\hspace{0.2cm}- \kappa \frac{\sin \mathcal{D}_0 }{\mathcal{D}_0}  \mathcal{D}(\theta(t))^{1-\alpha} \int_\Omega \psi (x_m,y ) \di y \nonumber\\
		&  \leq  \kappa \frac{\sin \mathcal{D}_0 }{\mathcal{D}_0} \mathcal{D}(\theta(t))^{1-\alpha} \left( \int_{|x_M - y|\geq |x_m - y|} (\psi(x_m, y) - \psi(x_M, y)) \di y - \int_\Omega \psi (x_m,y) \di y \right).
	\end{aligned}
	\end{align}
	Here, we used the following relation in the second inequality:
	\[  |x_M - y|\leq |x_m - y| \quad \Longrightarrow \quad \psi(x_m, y) - \psi(x_M, y)\le0. \]
Next, we combine Lemma \ref{L2.5} and \eqref{B-10} to obtain
\[
		\frac{\di}{\di t} \mathcal{D}(\theta(t)) 
		\leq \kappa  \frac{\sin \mathcal{D}_0 }{\mathcal{D}_0} \left(  \frac{2^{\beta}-1}{1-\beta} - \frac{1}{1-\beta} \right) \mathcal{D}(\theta(t))^{1-\alpha}
		= -\kappa  \frac{\sin \mathcal{D}_0 }{\mathcal{D}_0} \frac{2-2^{\beta}}{1-\beta}  \mathcal{D}(\theta(t))^{1-\alpha},
\]
i.e., 
\[
\frac{\di}{\di t} \mathcal{D}(\theta(t)) \leq  -\kappa  \frac{\sin \mathcal{D}_0 }{\mathcal{D}_0} \frac{2-2^{\beta}}{1-\beta}  \mathcal{D}(\theta(t))^{1-\alpha}, \quad t > 0. \]
We integrate the above inequality to find 
	\begin{align*}
		\mathcal{D}(\theta(t)) \leq \left(\mathcal{D}_0^{\alpha} - \alpha \kappa  \frac{\sin \mathcal{D}_0 }{\mathcal{D}_0}  \frac{2-2^{\beta}}{1-\beta} t \right)^{\frac{1}{\alpha}}.
	\end{align*}
By the comparison principle of ODE, there exists a positive time $t_c$ such that 
\[ t_c \geq  \frac{(1-\beta)\mathcal{D}_0^{1+\alpha} }{\alpha \kappa (\sin \mathcal{D}_0) (2-2^\beta)} \quad \mbox{and} \quad   \mathcal{D}(\theta(t))=0, \quad \mbox{for}~t \geq t_c.  \]
\end{proof}

\begin{remark}
Our result shows that complete phase synchronization occurs in finite time. This phenomenon is analogous to the results in \cite{P-P-J2021, Poyato2019}, where the corresponding discrete singular Kuramoto model with $\psi \equiv 1$ is studied. In our setting, the singularity of $\psi$ further enhances this effect, leading to even faster synchronization. In contrast, the work \cite{KoHa2025} investigates the SCKM with $\alpha = 0$ and a possibly integrable singularity in $\psi$, where it is shown that the phase diameter decays exponentially to zero. Nevertheless, finite-time complete phase synchronization does not occur in general in that case.
\end{remark}

\subsection{Nonidentical oscillators}\label{sec3.3}

In this subsection, we investigate the collective behavior of solutions $\theta(t,x)$ to \eqref{A-2} with $\nu \neq 0$. First, we show the boundedness of phase diameter.
\begin{lemma} \label{L3.3}
Suppose that system parameters, natural frequency function and initial data satisfy 
\begin{align*}
\begin{aligned}
& 0 <  \alpha, \beta < 1, \quad \Omega = [0, 1], \quad  \theta^{\mathrm{in}} \in L^2(\Omega), \\
& \kappa > \frac{\mathcal{D}(\nu) \mathcal{D}_0^\alpha (1-\beta)}{(\sin \mathcal{D}_0) (2-2^\beta)},  \quad  \nu \in B(\Omega), \quad 0<\mathcal{D}_0 <\pi, 
\end{aligned}
\end{align*}
and let $\theta = \theta(t,x)$ be a global solution to \eqref{A-2}. Then, the diameter functional ${\mathcal D}(\theta)$ remains bounded:
\[
\sup_{0 \leq t < \infty} \mathcal{D}(\theta(t)) \le \mathcal{D}_0 <\pi.
\]
\end{lemma}

\begin{proof}
	Note that we have 
	\[	\sup\limits_{0\le t<\infty}	\sup\limits_{x\in\Omega} \Big( \nu(x)+\kappa \displaystyle\int_{\Omega} \psi(x,y) h( \theta(y) - \theta(x))\di y \Big) \le \sup\limits_{x\in\Omega}\nu(x)+C_{\psi}\kappa.\]
	Suppose there exists first time $t_1 >0$ such that
	\begin{align} \label{C-7-00}
	\mathcal{D}(\theta(t_1)) = \mathcal{D}_0  \quad \mathrm{and} \quad \mathcal{D}(\theta(t))> \mathcal{D}_0 \quad \mathrm{for} \quad t \in (t_1, t_1 +\delta) \quad \mathrm{and} \ \mathrm{some} \ \delta>0. 
\end{align}		
	By Lemma \ref{L2.4}, for $t\leq t_1$ we have 
\[
		\frac{\sin\big( \theta(y) - \theta(x_m) \big)}{|\theta(y) - \theta(x_m)|^{\alpha}_o} \geq \frac{\sin\big( \mathcal{D}(\theta) \big)}{(\mathcal{D}(\theta))^{\alpha+1}} \big( \theta(y) - \theta(x_m) \big) \geq \frac{\sin\big( \mathcal{D}_0 \big)}{\mathcal{D}_0^{\alpha+1}} \big( \theta(y) - \theta(x_m) \big),
\]
	and
	\begin{align*}
	\begin{aligned}
		&\frac{\sin\big( \theta(y) - \theta(x_M) \big)}{|\theta(y) - \theta(x_M)|^{\alpha}_o} = 
		-\frac{\sin\big( \theta(x_M) - \theta(y) \big)}{|\theta(x_M) - \theta(y)|^{\alpha}_o}  \\
		&\hspace{1cm} \leq \frac{\sin\big(\mathcal{D}(\theta) \big)}{(\mathcal{D}(\theta))^{\alpha+1}} \big( \theta(y) - \theta(x_M) \big) 
		\leq \frac{\sin\big(\mathcal{D}_0 \big)}{\mathcal{D}_0^{\alpha+1}} \big( \theta(y) - \theta(x_M) \big).
	\end{aligned}
	\end{align*}
	We combine the above two estimates and a similar proof of \eqref{C-7} to see
	\begin{align*}
		& \frac{\di}{\di t} \mathcal{D}(\theta) = \nu(x_M)-\nu(x_m) \\
		&\hspace{0.2cm} +\kappa \displaystyle\int_{\Omega} \psi(x_M, y) \frac{\sin\big( \theta(y) - \theta(x_M) \big)}{|\theta(y) - \theta(x_M)|_o^{\alpha}}\di y - \kappa \displaystyle\int_{\Omega} \psi(x_m, y) \frac{\sin\big( \theta(y) - \theta(x_m) \big)}{|\theta(y) - \theta(x_m)|^{\alpha}_o} \di y \\
		&\leq \mathcal{D}(\nu)+ \kappa \frac{\sin \mathcal{D}_0}{\mathcal{D}_0^{\alpha+1}} \Big[ \int_{\Omega} \psi(x_M, y) \big( \theta(y) - \theta(x_M) \big) \di y- \int_{\Omega} \psi(x_m, y) \big( \theta(y) - \theta(x_m) \big) \di y \Big] \\
		&\leq \mathcal{D}(\nu)+ \kappa  \frac{\sin \mathcal{D}_0}{\mathcal{D}_0^{\alpha+1}} \Big[  \int_{\Omega} (\psi(x_M, y) - \psi(x_m, y))  \big( \theta(y) - \theta(x_M) \big) \di y- \int_{\Omega} \psi(x_m, y)  \big( \theta(x_M) - \theta(x_m) \big) \di y \Big] \\
		&=\mathcal{D}(\nu)+  \kappa \frac{\sin \mathcal{D}_0}{\mathcal{D}_0^{\alpha+1}}  \Big[ \int_{|x_M - y|\geq |x_m - y|} (\psi(x_M, y) - \psi(x_m, y))  \big( \theta(y) - \theta(x_M) \big) \di y \\
		&\hspace{0.2cm}+  \int_{|x_M - y|< |x_m - y|} (\psi(x_M, y) - \psi(x_m, y))  \big( \theta(y) - \theta(x_M) \big) \di y - \int_{\Omega} \psi(x_m, y)  \big( \theta(x_M) - \theta(x_m) \big) \di y \Big] \\
		&\leq \mathcal{D}(\nu)+  \kappa \frac{\sin \mathcal{D}_0}{\mathcal{D}_0^{\alpha+1}} \Big[  \int_{|x_M - y|\geq |x_m - y|} (\psi(x_m, y) - \psi(x_M, y))  \big( \theta(x_M) - \theta(y) \big) \di y - \mathcal{D}(\theta) \displaystyle\int_{\Omega} \psi(x_m, y) \di y \Big] \\
		&\leq \mathcal{D}(\nu)+  \kappa \frac{\sin \mathcal{D}_0}{\mathcal{D}_0^{\alpha+1}}  \mathcal{D}(\theta) \Big [ \int_{|x_M - y|\geq |x_m - y|} (\psi(x_m, y) - \psi(x_M, y)) \di y - \min_{x\in \Omega}\displaystyle\int_{\Omega} \psi(x, y)  \di y \Big] .
	\end{align*}
	With Lemma \ref{L2.5} and \eqref{B-10}, one has 
	\[ 
	\frac{\di}{\di t} \mathcal{D}(\theta) \leq \mathcal{D}(\nu)+   \kappa \frac{\sin \mathcal{D}_0}{\mathcal{D}_0^{\alpha+1}}  \left(  \frac{2^\beta -1}{1-\beta} - \frac{1}{1-\beta}\right) \mathcal{D}(\theta) \leq \mathcal{D}(\nu)-  \kappa \frac{\sin \mathcal{D}_0}{\mathcal{D}_0^{\alpha+1}}   \left(  \frac{2 -2^\beta}{1-\beta}\right) \mathcal{D}(\theta),
	\]
i.e.,
\begin{equation} \label{new419}
			\frac{\di}{\di t} \mathcal{D}(\theta) \Big\vert_{t=t_1} \leq \mathcal{D}(\nu)-  \kappa \frac{\sin \mathcal{D}_0}{\mathcal{D}_0^{\alpha+1}}   \left(  \frac{2 -2^\beta}{1-\beta}\right) \mathcal{D}_0 <0,
	\end{equation}
	where we use the condition on $\kappa$. This contradicts the assumption \eqref{C-7-00} and proves the claim. 
\end{proof}

Next, we present the main result on the practical synchronization as follows. 
\begin{theorem}  \label{T3.3}
Suppose that system parameters, natural frequency function and initial data satisfy 
\[  0 <  \alpha, \beta < 1,  \quad \Omega = [0, 1], \quad  \theta^{\mathrm{in}} \in L^2(\Omega), \quad  \nu \in B(\Omega), \quad 0<\mathcal{D}_0 <\pi,
\]
and let $\theta = \theta(t,x)$ be a global solution to \eqref{A-2}.  Then for $\varepsilon \in (0, 1)$, there exists a positive constant $\kappa_*(\varepsilon)$  such that if 
\begin{equation} \label{C-7-1}
\kappa>\kappa_*(\varepsilon):= \max \left\lbrace \frac{\mathcal{D}(\nu) \mathcal{D}_0^{\alpha+1} (1-\beta)}{\varepsilon  (\sin \mathcal{D}_0) (2-2^\beta)} , ~ \frac{\mathcal{D}(\nu) \mathcal{D}_0^\alpha (1-\beta)}{(\sin \mathcal{D}_0) (2-2^\beta)} \right\rbrace,
\end{equation}
there exists some finite time $t_* >0$ such that
\begin{align*}
\mathcal{D}(\theta(t)) <\varepsilon \quad \forall~t\geq t_*.
\end{align*}
\end{theorem}

\begin{proof} Let $\varepsilon \in (0, 1)$ be an arbitrary constant and we assume that 
\[ \kappa> \kappa_*(\varepsilon). \]
By Lemma \ref{L3.3}, we have
\begin{equation} \label{C-8}
\sup_{0 \leq t < \infty} \mathcal{D}(\theta(t)) \leq \mathcal{D}_0 <\pi.
\end{equation}
Now, we consider two cases:
\[  \mbox{Either} \quad \mathcal{D}_0 < \varepsilon \quad \mbox{or} \quad \mathcal{D}_0 \geq \varepsilon.  \]
\noindent $\bullet$~Case A ($ \mathcal{D}_0 < \varepsilon$): It follows from \eqref{C-8} that 
\[  t_* = 0, \quad \mathcal{D}(\theta(t)) \leq   \mathcal{D}_0 < \varepsilon, \quad \forall~t \geq t_* = 0, \]
which gives the desired estimate. \newline

\noindent $\bullet$~Case B ($\mathcal{D}_0 \geq \varepsilon$): As in \eqref{new419}, we have
\begin{equation} \label{C-9}
\frac{\di}{\di t} \mathcal{D}(\theta (t)) \leq  \mathcal{D}(\nu)-      \kappa \frac{\sin \mathcal{D}_0}{\mathcal{D}_0^{\alpha+1}}  \left(  \frac{2 -2^\beta}{1-\beta}\right) \mathcal{D}(\theta (t)).
\end{equation}
Now, we consider the corresponding linear differential equation:
\[ 
\begin{cases}
\displaystyle \frac{\di}{\di t}y(t)=\mathcal{D}(\nu)-  \kappa \frac{\sin \mathcal{D}_0}{\mathcal{D}_0^{\alpha+1}} \left(  \frac{2 -2^\beta}{1-\beta}\right) y(t), \quad t > 0, \\
\displaystyle y \Big|_{t = 0} = \mathcal{D}_0.
\end{cases}
\]
This yields
\begin{align}
\begin{aligned} \label{C-10}
y(t) = \left(  \mathcal{D}_0- \frac{\mathcal{D}(\nu) \mathcal{D}_0^{\alpha +1}(1-\beta)}{\kappa (\sin \mathcal{D}_0) (2-2^\beta)} \right) \exp \left( - \kappa \frac{\sin \mathcal{D}_0}{\mathcal{D}_0^{\alpha+1}} \left(  \frac{2 -2^\beta}{1-\beta}\right)  t \right) + \frac{\mathcal{D}(\nu) \mathcal{D}_0^{\alpha +1}(1-\beta)}{\kappa (\sin \mathcal{D}_0) (2-2^\beta)}.
\end{aligned}
\end{align}
By the comparison principle of ODE together with \eqref{C-9} and \eqref{C-10}, we have 
\begin{align} \label{D-4-14}
\begin{aligned}
\mathcal{D}(\theta(t)) &\leq  \left(  \mathcal{D}_0 - \frac{\mathcal{D}(\nu) \mathcal{D}_0^{\alpha +1}(1-\beta)}{\kappa (\sin \mathcal{D}_0) (2-2^\beta)} \right) \exp \left( - \kappa \frac{\sin \mathcal{D}_0 }{\mathcal{D}_0^{\alpha+1}} \left(  \frac{2 -2^\beta}{1-\beta}\right)  t \right)  \\
& \hspace{0.2cm} + \frac{\mathcal{D}(\nu) \mathcal{D}_0^{\alpha +1}(1-\beta)}{\kappa (\sin \mathcal{D}_0) (2-2^\beta)}.
\end{aligned}
\end{align}
On the other hand, it follows from the condition on $\kappa$ that 
\begin{equation} \label{D-4-14-1}
\varepsilon - \frac{\mathcal{D}(\nu) \mathcal{D}_0^{\alpha +1}(1-\beta)}{\kappa (\sin \mathcal{D}_0) (2-2^\beta)} >0.
\end{equation} 
Since 
\begin{align*}
\lim_{t\rightarrow \infty} \exp \left( -  \kappa  \frac{\sin \mathcal{D}_0}{\mathcal{D}_0^{\alpha+1}} \left(  \frac{2 -2^\beta}{1-\beta}\right)  t \right) = 0,
\end{align*}
there exists a sufficiently large  $t_*>0$ such that
\begin{equation} \label{D-4-15}
\left( \mathcal{D}_0 - \frac{\mathcal{D}(\nu) \mathcal{D}_0^{\alpha +1}(1-\beta)}{\kappa (\sin \mathcal{D}_0) (2-2^\beta)} \right) \exp \left( -  \kappa  \frac{\sin \mathcal{D}_0}{\mathcal{D}_0^{\alpha+1}} \left(  \frac{2 -2^\beta}{1-\beta}\right)  t \right) \leq \varepsilon - \frac{\mathcal{D}(\nu) \mathcal{D}_0^{\alpha +1}(1-\beta)}{\kappa (\sin \mathcal{D}_0) (2-2^\beta)}, \quad t\geq t_*.
\end{equation}
Then, for $t \geq t_*$, we combine \eqref{D-4-14} and \eqref{D-4-15} to obtain
\[
\mathcal{D} (\theta(t))\leq  \varepsilon -  \frac{\mathcal{D}(\nu) \mathcal{D}_0^{\alpha +1}(1-\beta)}{\kappa (\sin \mathcal{D}_0) (2-2^\beta)} + \frac{\mathcal{D}(\nu) \mathcal{D}_0^{\alpha +1}(1-\beta)}{\kappa (\sin \mathcal{D}_0) (2-2^\beta)} = \varepsilon, \quad t\geq t_*.
\]
\end{proof}

\begin{remark} \label{R3.4}
\begin{enumerate} Below, we provide several comments on the result of Theorem \ref{T3.3}. 
\item[(i)]
For Case B in the proof of Theorem \ref{T3.3}:
\begin{equation} \label{D-4-15-1}
 \mathcal{D}_0 \geq \varepsilon,
\end{equation}
we can provide a quantitative estimate for $t_* >0$ as follows. The assumption on the coupling strength yields
\[ \varepsilon > \frac{\mathcal{D}(\nu) \mathcal{D}_0^{\alpha +1}(1-\beta)}{ \kappa (\sin \mathcal{D}_0) (2-2^\beta)}. \]
Therefore, with this and \eqref{D-4-15-1} it follows
\[ 
{\mathcal D}_0  > \frac{\mathcal{D}(\nu) \mathcal{D}_0^{\alpha +1}(1-\beta)}{ \kappa (\sin \mathcal{D}_0) (2-2^\beta)}.
\]
Note that 
\begin{align*}
\begin{aligned}
& \exp \left[  - \kappa \frac{\sin \mathcal{D}_0}{\mathcal{D}_0^{\alpha+1}} \left(  \frac{2 -2^\beta}{1-\beta}\right)  t \right ] \leq \frac{\varepsilon - \frac{\mathcal{D}(\nu) \mathcal{D}_0^{\alpha +1}(1-\beta)}{\kappa (\sin \mathcal{D}_0) (2-2^\beta)}}{ \mathcal{D}_0 - \frac{\mathcal{D}(\nu) \mathcal{D}_0^{\alpha +1}(1-\beta)}{\kappa (\sin \mathcal{D}_0) (2-2^\beta)} } \\
& \hspace{0.5cm} \iff \quad  t\geq t_*:= \frac{\mathcal{D}_0^{\alpha+1}(1-\beta)}{\kappa (\sin \mathcal{D}_0) (2-2^\beta)} \ln \left(  \frac{\mathcal{D}_0 - \frac{\mathcal{D}(\nu) \mathcal{D}_0^{\alpha +1}(1-\beta)}{\kappa (\sin \mathcal{D}_0 ) (2-2^\beta)}}{\varepsilon - \frac{\mathcal{D}(\nu) \mathcal{D}_0^{\alpha +1}(1-\beta)}{\kappa (\sin \mathcal{D}_0) (2-2^\beta)}}  \right).
\end{aligned}
\end{align*}
\item[(ii)]
The corresponding discrete singular Kuramoto model with $\psi \equiv 1$ is analyzed in references \cite{P-P-J2021, Poyato2019}. They established asymptotic synchronization by specific structure and ordering of the discrete model, which does not carry over to our continuous setting. On the other hand, the authors in \cite{KoHa2025} obtain practical synchronization for SCKM with $\alpha = 0$ under stronger assumptions on the initial diameter and the coupling strength. This further highlights the crucial sticking role played by the singularity of $h$.
\end{enumerate}
\end{remark}

\vspace{0.5cm}

Next, we provide an interesting property of the SCKM \eqref{A-2} which are similar to order preservation results in the Kuramoto model.
\begin{proposition} \label{P3.3}
\emph{(Propagation of monotonicity)}
Suppose that the natural frequency function and initial data 
\begin{align*}
\nu : \Omega \to \bbr \quad \mathrm{bounded} \ \mathrm{and} \quad \theta^\mathrm{in} : \Omega \to \bbr \quad \mathrm{on} \quad  \Omega = [0, 1],
\end{align*}
are square-integrable and defined pointwise. Moreover, suppose
\begin{align*}
\begin{aligned}
& 0 <  \alpha, \beta < 1, \quad  \kappa > \frac{\mathcal{D}(\nu) \mathcal{D}_0^\alpha (1-\beta)}{(\sin\mathcal{D}_0)(2-2^\beta)}, \quad 0<\mathcal{D}_0 <\pi  \\
&  (\nu(y) -\nu(x)) (y - x) > 0 \quad \mathrm{and} \quad  (\theta^{\mathrm{in}}(y) -\theta^{\mathrm{in}}(x)) (y - x) > 0, \quad \forall~x \neq y,
\end{aligned}
\end{align*}
and let $\theta = \theta(t,x)$ be a global solution to \eqref{A-2}. Then for $t > 0$, $\theta(t, \cdot)$ is strictly increasing in the sense that 
\[ (\theta(t,y) -\theta(t,x)) (y - x) > 0, \quad \forall~x \neq y. \]
\end{proposition}
\begin{proof}
Due to the assumption on the coupling strength and initial diameter, we can apply Lemma \ref{L3.3} to obtain
\begin{align*}
\mathcal{D}(\theta(t)) < \pi \quad \mathrm{for} \ t\geq 0.
\end{align*}
For $t > 0$, let $x_1,x_2 \in \Omega$ such that 
\[  x_1 < x_2. \]
Then, we define the subsets of $\Omega$:
\begin{align*}
\begin{aligned}
&A(t):= \lbrace y \in \Omega : \theta (t,y)\leq \theta(t,x_2) \rbrace, \quad B(t):= \lbrace y \in \Omega : \theta (t,y)\geq \theta(t,x_1) \rbrace, \\
&C(t) := \left\lbrace y \in \Omega : \frac{\sin (\theta(t,y) - \theta(t,x_2))}{|\theta(t,y) - \theta(t,x_2)|_o^\alpha} - \frac{\sin (\theta(t,y) - \theta(t,x_1))}{|\theta(t,y) - \theta(t,x_1)|_o^\alpha} \leq 0 \right\rbrace.
\end{aligned}
\end{align*}
By the monotonicity of $\theta^{\mathrm{in}}$, we have
\[    \theta^{\mathrm{in}}(x_1) < \theta^{\mathrm{in}} (x_2). \]
 Moreover, we combine  \eqref{B-9} with Lemma \ref{L2.3} to find
\begin{align*}
\begin{aligned}
&\partial_t (\theta (x_2) - \theta (x_1)) = \nu(x_2) - \nu(x_1)  \\
& \hspace{0.5cm} + \kappa \int_\Omega \psi(x_2,y) \frac{\sin (\theta(y) - \theta(x_2))}{|\theta(y) - \theta(x_2)|^\alpha_o} \di y - \kappa \int_\Omega \psi(x_1,y) \frac{\sin (\theta(y) - \theta(x_1))}{|\theta(y) - \theta(x_1)|^\alpha_o} \di y \\
& \hspace{0.3cm} = \nu(x_2) - \nu(x_1) + \kappa \int_A \psi(x_2,y) \frac{\sin (\theta(y) - \theta(x_2))}{|\theta(y) - \theta(x_2)|_o^\alpha} \di y \\
& \hspace{0.5cm}- \kappa \int_B \psi(x_1,y) \frac{\sin (\theta(y) - \theta(x_1))}{|\theta(y) - \theta(x_1)|_o^\alpha} \di y + \kappa \int_{\Omega\setminus A} \psi(x_2,y) \frac{\sin (\theta(y) - \theta(x_2))}{|\theta(y) - \theta(x_2)|_o^\alpha} \di y \\
& \hspace{0.5cm}- \kappa \int_{\Omega \setminus B} \psi(x_1,y) \frac{\sin (\theta(y) - \theta(x_1))}{|\theta(y) - \theta(x_1)|_o^\alpha} \di y \\
&  \hspace{0.3cm} \geq \nu(x_2) - \nu(x_1) \\
&  \hspace{0.5cm} + \kappa \int_{A\cap B} (\psi(x_1,y) + \psi (x_2,y)) \left(\frac{\sin (\theta(y) - \theta(x_2))}{|\theta(y) - \theta(x_2)|_o^\alpha} - \frac{\sin (\theta(y) - \theta(x_1))}{|\theta(y) - \theta(x_1)|_o^\alpha} \right) \di y \\
& \hspace{0.5cm} + \kappa \int_{A\setminus B} \psi(x_2,y) \frac{\sin (\theta(y) - \theta(x_2))}{|\theta(y) - \theta(x_2)|^\alpha_o} \di y - \kappa \int_{B \setminus A} \psi(x_1,y) \frac{\sin (\theta(y) - \theta(x_1))}{|\theta(y) - \theta(x_1)|^\alpha_o} \di y \\
& \hspace{0.5cm} + \kappa \int_{\Omega\setminus A} \psi(x_2,y) \frac{\sin (\theta(y) - \theta(x_2))}{|\theta(y) - \theta(x_2)|_o^\alpha} \di y - \kappa \int_{\Omega \setminus B} \psi(x_1,y) \frac{\sin (\theta(y) - \theta(x_1))}{|\theta(y) - \theta(x_1)|_o^\alpha} \di y \\
&  \hspace{0.3cm} \geq  \nu(x_2) - \nu(x_1) \\
&  \hspace{0.5cm}+ \kappa \int_{A\cap B} (\psi(x_1,y) + \psi (x_2,y)) \left(\frac{\sin (\theta(y) - \theta(x_2))}{|\theta(y) - \theta(x_2)|_o^\alpha} - \frac{\sin (\theta(y) - \theta(x_1))}{|\theta(y) - \theta(x_1)|_o^\alpha} \right) \di y \\
& \hspace{0.5cm}  + \kappa \int_{A\setminus B} \psi(x_2,y) \frac{\sin (\theta(y) - \theta(x_2))}{|\theta(y) - \theta(x_2)|^\alpha_o} \di y - \kappa \int_{B \setminus A} \psi(x_1,y) \frac{\sin (\theta(y) - \theta(x_1))}{|\theta(y) - \theta(x_1)|^\alpha_o} \di y \\
& \hspace{0.5cm} + \kappa \int_{B\setminus A} \psi(x_2,y) \frac{\sin (\theta(y) - \theta(x_2))}{|\theta(y) - \theta(x_2)|^\alpha_o} \di y - \kappa \int_{A \setminus B} \psi(x_1,y) \frac{\sin (\theta(y) - \theta(x_1))}{|\theta(y) - \theta(x_1)|^\alpha_o} \di y \\
\end{aligned}
\end{align*}
\begin{align*}
\begin{aligned}
&  \hspace{0.3cm} \geq  \nu(x_2) - \nu(x_1) \\
& \hspace{0.5cm} + \kappa \int_{A\cap B} (\psi(x_1,y) + \psi (x_2,y)) \left(\frac{\sin (\theta(y) - \theta(x_2))}{|\theta(y) - \theta(x_2)|_o^\alpha} - \frac{\sin (\theta(y) - \theta(x_1))}{|\theta(y) - \theta(x_1)|_o^\alpha} \right) \di y \\
& \hspace{0.5cm}+ \kappa \int_{(A\setminus B)\cap C} (\psi(x_1,y) + \psi (x_2,y)) \left(\frac{\sin (\theta(y) - \theta(x_2))}{|\theta(y) - \theta(x_2)|_o^\alpha} - \frac{\sin (\theta(y) - \theta(x_1))}{|\theta(y) - \theta(x_1)|_o^\alpha} \right) \di y \\
& \hspace{0.5cm}+ \kappa \int_{(B\setminus A)\cap C} (\psi(x_1,y) + \psi (x_2,y)) \left(\frac{\sin (\theta(y) - \theta(x_2))}{|\theta(y) - \theta(x_2)|^\alpha_o} - \frac{\sin (\theta(y) - \theta(x_1))}{|\theta(y) - \theta(x_1)|^\alpha_o} \right) \di y  \\
&  \hspace{0.3cm} \geq  \nu(x_2) - \nu(x_1)  - 6\kappa \frac{2^\beta-1}{1-\beta} |\theta(x_1) - \theta(x_2)|^{1-\alpha},
\end{aligned}
\end{align*}
i.e., we have
\[
\partial_t (\theta (x_2) - \theta (x_1)) \geq  \nu(x_2) - \nu(x_1)  - 6\kappa \frac{2^\beta-1}{1-\beta} |\theta(x_1) - \theta(x_2)|^{1-\alpha}.
\]
Suppose that there exists first time $t_* >0$ such that
\begin{align*}
\theta (t,x_1)  < \theta (t,x_2) \quad \mathrm{for} \ t< t_*, \quad&\theta (t_*,x_2) =  \theta (t_*,x_1).
\end{align*}
This yields
\begin{equation} \label{C-29}
\partial_t (\theta (t,x_2) - \theta (t,x_1)) \Big\vert_{t=t_*} \geq \nu(x_2) - \nu(x_1).
\end{equation}
Since the natural frequency function $\nu$ is strictly increasing, we have
\begin{equation} \label{C-30}
\nu(x_2) - \nu(x_1)>0.
\end{equation}
Thus, it follows from \eqref{C-29} and \eqref{C-30} that $(\theta (t,x_2) - \theta (t,x_1))$ is in strictly increasing mode, i.e. there exists a constant $0 <  \delta \ll 1$ such that 
\[ \theta (t,x_1) > \theta (t,x_2), \quad t_*-\delta < t  < t_*, \]
which is contradictory to the choice of $t_*$. 
\end{proof}

\begin{remark}\label{R4.3}
Note that unlike in the finite-particle system, we cannot obtain a uniform lower bound on the derivative $\partial_t (\theta(t,x_2) - \theta(t,x_1))$ because it is possible that $\nu(x_2) - \nu(x_1) \rightarrow 0$ for $x_2 - x_1 \rightarrow 0$. That means that, although two oscillators never collide, their distance can become arbitrarily small if they start close. If we put some stronger assumptions on the function $\nu$, we can obtain a lower bound that depends on the distance between $x_1$ and $x_2$. For example, if we assume that $\nu$ is differentiable and has uniform growth, i.e. $\nu' \geq m >0$, we observe
\begin{align*}
\nu(x_2)- \nu(x_1) = \nu' (\xi) (x_2-x_1) \geq m(x_2-x_1),
\end{align*}
where $\xi\in (x_1,x_2)$. 
\end{remark}
\section{From particle model to the SCKM} \label{sec:4}
\setcounter{equation}{0}
In this section, we study the derivation of the SCKM from the particle model via the graph limit on the bounded domain $\Omega = [0, 1]$ equipped with the Lebesgue measure. \newline

\subsection{Construction of approximate solution} \label{sec:4.1}
Consider the particle model:
\begin{equation} \label{D-1}
\begin{cases}
		\dot{\theta}_i(t) = \nu_i + \displaystyle\frac{\kappa}{N} \sum_{k=1}^{N} \psi_{ik}^{(N)} h\big(\theta_{k}(t) - \theta_i(t) \big) ,\quad & t > 0, \\
		\theta_i(0)=\theta^{\mathrm{in}}_i, &i\in[N],
		\end{cases}
\end{equation}
and the continuum model:
\begin{equation} \label{D-2}
	\begin{cases}
		\partial_t \theta(t, x) = \nu(x) + \kappa \displaystyle\int_{\Omega} \psi(x, y) h\big( \theta(t, y) - \theta(t, x) \big)  \di y, & t > 0, \; x \in \Omega, \\
		\theta(0, x) = \theta^{\mathrm{in}}(x), & x \in \Omega.
	\end{cases}
\end{equation}
First, we partition the domain $\Omega$ into $N$ equal subintervals:
\[
I_1:= \big[0, 1/N \big], \quad I_i:=\big((i-1)/N,\, i/N\big],\quad i \in \{ 2,\dots,N \}, \quad [0, 1]  = \bigcup_{i=1}^{N} I_i.
\]
With this partition $\{ I_i \}$, we define the following step-function (piecewise-constant) interpolants:
\[
\begin{cases}
\displaystyle \theta^N(t,x):=\theta_i(t),\quad \nu^N(x):=\nu_i\ & \text{ for } t > 0, \;x\in I_i, \\
\displaystyle \psi^N(x,y):=\psi_{ik}^{(N)}\ &\text{ for }(x,y)\in I_i\times I_k.
\end{cases}
\]
With these definitions, the Cauchy problem \eqref{D-1} can be rewritten as the Cauchy problem for the step-graphon evolutionary equation (singular Kuramoto model, or SKM, in short):
\begin{equation} 
\begin{cases} \label{D-3}
\displaystyle \partial_t \theta^N(t,x)
	\;=\; \nu^N(x) \;+\; \kappa\int_{\Omega} \psi^N(x,y)\, h\big(\theta^N(t,y)-\theta^N(t,x)\big)\di y,
	\quad t>0,\ x\in\Omega, \\
\displaystyle \theta^N(0,x)=\theta^{\mathrm{in},N}(x):=\theta^{\mathrm{in}}_i, \quad x\in I_i. 
\end{cases}
\end{equation}

\subsection{Derivation of the SCKM} \label{sec:4.2}
In this subsection, we study the rigorous derivation of the SCKM using the graph limit. For this, we consider a sufficient framework $({\mathcal F})$ in terms of system parameters and initial data. Under this framework, we can ensure well-posedness and convergence of the step-graphon approximations. More precisely, we list conditions in the framework  $({\mathcal F})$ as follows. 
\vspace{0.1cm}
\begin{itemize}
\item 
$({\mathcal F}_1)$:~The communication weight function and its discretization satisfy
\[ \psi\in L^1(\Omega^2) \quad \mbox{and} \quad \psi^N\to\psi \quad \mbox{in} \quad L^1(\Omega^2) \quad \mbox{as $N\to\infty$}, \]
where $\psi^N$ are cell averages of $\psi$ on $I_i\times I_k$.
\vspace{0.2cm}
\item
$({\mathcal F}_2)$:~The natural frequency function and its discretization satisfy
\[  \nu\in L^2(\Omega) \quad \mbox{and} \quad  \nu^N\to \nu \quad \mbox{in $L^1(\Omega)$} \quad \mbox{as $N\to\infty$}.
 \]
 \item
$({\mathcal F}_3)$:~The initial data satisfies
\[ \theta^{\mathrm{in}}\in L^2(\Omega) \quad \mbox{and} \quad \theta^{\mathrm{in},N} \to \theta^{\mathrm{in}} \quad \mbox{in $L^2(\Omega)$} \quad \mbox{as $N\to\infty$}. \]
\item
$({\mathcal F}_4)$:~System parameters and $h$ satisfy 
\[  \beta\in [0,1), \quad \alpha\in[0,1), \quad \| h \|_{\infty} \leq 1, \quad |h(a)-h(b)|\le C_\alpha |a-b|^{1-\alpha}~~\mbox{on $[-2\pi,2\pi]$}. \]
\end{itemize}
\begin{remark}
We consider the setting where
\begin{align*}
\theta^\mathrm{in} \in L^2(\Omega), \quad \nu \in L^2 (\Omega),
\end{align*}
and $h$, $\psi$, $\alpha$ and $\beta$ are given in \eqref{A-3}. In this setting, the assumptions of the framework $(\mathcal{F})$ are satisfied by construction of the discretization \eqref{D-3}.
\end{remark}
Next, we study an error estimate for the convergence of $\theta^N$ to $\theta$. 
\begin{lemma} \label{L4.2} 
Suppose the framework $({\mathcal F})$ holds. Then, for every $T \in (0, \infty]$ there exist a positive constant $C_T$ and unique solutions $\theta^N$ and $\theta$ to \eqref{D-3} and \eqref{D-2}, respectively, such that
\begin{align} \label{D-4}
\frac{\di }{\di t}	\| 	\theta^{N} - \theta \|_{2}^{2} \leq C_T \Big( \|\nu^N -\nu \|_{1}+ \|\psi^{N}-\psi\|_{1}+ \| 	\theta^{N} - \theta \|_{2}^{2} \Big).
\end{align}
\end{lemma}
\begin{proof}
Note that $-h$ is one-sided Lipschitz continuous by Lemma \ref{P4.1}. Hence, the right-hand sides of \eqref{D-1} and \eqref{D-3} are uniformly bounded and satisfy the regularity condition, yielding global and unique Filippov solutions. We refer to \cite{P-P-J2021} for more details on these results. \\

By direct calculation, one has 
	\begin{align}\label{D-5}
		\begin{aligned}
			&	\frac{\di }{\di t}	\| 	\theta^{N}(t) - \theta (t) \|_{2}^{2} = 2\int_{\Omega} (\theta^{N}(x) - \theta (x)) (\partial_t\theta^{N}(x) - \partial_t\theta (x) ) \di x\\
			& \hspace{0.5cm} = 2\int_{\Omega} (\theta^{N}(x) - \theta (x)) (\nu^N(x)-\nu(x)) \di x\\
			&\hspace{0.7cm}  + 2\iint_{\Omega^2} (\theta^{N}(x) - \theta (x)) (\psi^{N} (x,y) h(\theta^{N}(y)-\theta^{N}(x))-\psi(x,y)h(\theta(y)-\theta(x))) \di y\di x\\
			&\hspace{0.5cm}  = 2\int_{\Omega} (\theta^{N}(x) - \theta (x)) (\nu^N(x)-\nu(x)) \di x\\
			&\hspace{0.7cm}  +2\iint_{\Omega^2} (\theta^{N}(x) - \theta (x)) (\psi^{N} (x,y) h(\theta^{N}(y)-\theta^{N}(x))-\psi(x,y) h(\theta^{N}(y)-\theta^{N}(x))) \di y \di x\\
			&\hspace{0.7cm} +2\iint_{\Omega^2} (\theta^{N}(x) - \theta (x)) (\psi (x,y)h(\theta^{N}(y)-\theta^{N}(x))-\psi(x,y)h(\theta(y)-\theta(x))) \di y\di x \\
			&\hspace{0.5cm}  =:\mathcal{I}_{21}+\mathcal{I}_{22}+\mathcal{I}_{23}.
		\end{aligned}
	\end{align}
In the sequel, we estimate the terms ${\mathcal I}_{2i},~i = 1, 2,3$ one by one. \newline

\noindent $\bullet$~Estimate of ${\mathcal I}_{21}$: Note that 
	\begin{align}\label{D-6}
	\begin{aligned}
			\mathcal{I}_{21} &=2	\int_{\Omega} (\theta^{N}(t,x) - \theta (t,x)) (\nu^N(x)-\nu(x)) \di x \\
			&\le  2\int_{\Omega}|\theta^{N}(t,x) - \theta (t,x)|| \nu^N(x)-\nu(x)| \di x \le 4\pi\|\nu^N-\nu\|_{1}.
		\end{aligned}
	\end{align} \\
	
\noindent $\bullet$~Estimate of ${\mathcal I}_{22}$:~We use the bound $\|h\|_{\infty}\le 1$ to see 
	\begin{align}\label{D-7}
		\begin{aligned}
			\mathcal{I}_{22} &= 2\iint_{\Omega^2}  (\theta^{N}(x) - \theta (x)) (\psi^{N} (x,y) - \psi(x,y)) h(\theta^{N}(y)-\theta^{N}(x)) \di x\di y\\
			&	\le 4\pi \iint_{\Omega^2}|\psi^{N} (x,y)-\psi(x,y)|\di x\di y.
		\end{aligned}
	\end{align} \\
	
 \noindent $\bullet$~Estimate of ${\mathcal I}_{23}$:~We use the one-sided Lipschitz continuity of $-h$ from Lemma \ref{P4.1} and $h(-z)=-h(z)$ to estimate 
	\begin{align}\label{D-8}
			\mathcal{I}_{23}&= 2\iint_{\Omega^2}\psi(x,y) (\theta^{N}(x) - \theta (x))( h(\theta^{N}(y)-\theta^{N}(x))-h(\theta(y)-\theta(x))) \di x \di y \nonumber \\
			&=- 2\iint_{\Omega^2}\psi(y,x)(\theta^{N}(y) - \theta (y))( h(\theta^{N}(y)-\theta^{N}(x))-h(\theta(y)-\theta(x)))\di x \di y \nonumber\\
			&= \iint_{\Omega^2}\psi(x,y)\left(\left(\theta(y)- \theta (x)\right)-\left(\theta^{N}(y)- \theta^{N}(x)\right)\right) \nonumber\\
			&\qquad \times\left( -h(\theta(y)-\theta(x))-\left(-h(\theta^{N}(y)-\theta^{N}(x))\right)\right) \di x\di y \nonumber\\
			&\le  \iint_{\Omega^2}\psi(x,y)L_{h}|\left(\theta(y)- \theta (x)\right)-\left(\theta^{N}(y)- \theta^{N}(x)\right)|^2\di x\di y \nonumber\\
			&\le 2\iint_{\Omega^2}\psi(x,y)L_{h}\left(|\theta(y)-\theta^{N}(y)|^2+|\theta (x)-\theta^{N}(x)|^2\right)\di x\di y \nonumber\\
			&=4 \iint_{\Omega^2}\psi(x,y)L_{h}|\theta(y)-\theta^{N}(y)|^2\di x\di y \nonumber \\ 
			&\leq	4 C_\psi L_h\| 	\theta^{N}(t) - \theta (t) \|_{2}^{2}.
	\end{align}
In \eqref{D-5}, we collect all the estimates in \eqref{D-6}, \eqref{D-7}, \eqref{D-8} to find the desired estimate:
	\begin{align*}
		\begin{aligned}
			& \frac{\di }{\di t}	\| 	\theta^{N} - \theta \|_{2}^{2} \\
			& \hspace{0.5cm} \le 4\pi \|\nu^N -\nu \|_{1}+4\pi \iint_{\Omega^2}|\psi^{N} (x,y)-\psi(x,y)|\di x\di y +	4 C_\psi L_h\| 	\theta^{N} - \theta \|_{2}^{2}\\
			& \hspace{0.5cm}=:C_T \left( \|\nu^N -\nu \|_{1}+\|\psi^{N}-\psi\|_{1}+\| 	\theta^{N} - \theta \|_{2}^{2} \right).
		\end{aligned}
	\end{align*}		
\end{proof}

\begin{theorem} 
\emph{(Graph limit of the SKM)} \label{T4.1}
Suppose that the framework $({\mathcal F})$ holds and $T \in (0, \infty]$, let $\theta^N$ and $\theta$ be solutions to \eqref{D-3} and \eqref{D-2} in the time interval $[0, T)$, respectively. Then, there exists a positive constant ${\tilde C}_T$ such that 
\begin{align}
\begin{aligned} \label{D-9}
& (i)~\sup_{0 \leq t < T} \|\theta^N(t)-\theta(t)\|_{2}^2
	\;\le\; {\tilde C}_T\Big(
	\|\theta^{\mathrm{in},N}-\theta^{\mathrm{in}}\|_{2}^2
	+\|\nu^N-\nu\|_{1}
	+\|\psi^N-\psi\|_{1}
	\Big). \\
& (ii)~\theta^N \to \theta \quad \text{in } C\big([0,T);L^2(\Omega)\big)\quad\text{as }N\to\infty.
\end{aligned}
\end{align}
\end{theorem}
\begin{proof} (i)~We apply  Gr\"onwall's lemma to \eqref{D-4} for the error estimate:
\[ \sup_{0\le t\le T}\bigl\|\theta^{N}(t)-\theta(t)\bigr\|_{2}^{2}
		\le \tilde{C}_T\Big(\|\theta^{N, \mathrm{in}}-\theta^{\mathrm{in}} \|_{2}^{2}+\|\nu^N-\nu\|_{1}+\|\psi^{N}-\psi\|_{1}\Big).
\]
We use the same estimate to derive the uniqueness of solutions. Indeed,  we apply similar arguments to two solutions $\theta$ and $\tilde{\theta}$ with the same initial data $\theta^{\mathrm{in}}=\tilde{\theta}^{\mathrm{in}}$ to find 
	\begin{equation} \label{D-10}
	\sup_{0\le t\le T}\bigl\|\theta(t)-\tilde{\theta}(t)\bigr\|_{2}^{2} \le \tilde{C}_T\|\theta^{\mathrm{in}} -\tilde{\theta}^{\mathrm{in}} \|_{2}^{2}=0, \quad \text{for any} \quad T>0.
	\end{equation}
This yields uniqueness.  \newline

\noindent (ii)~It is easy to see that the framework $({\mathcal F}_1) – ({\mathcal F}_4)$ and \eqref{D-10} imply the desired convergence estimate. 
\end{proof}
As a direct corollary of Theorem \ref{T4.1}, we have the uniform-in-time graph limit for homogeneous case. 
\begin{corollary}\label{C4.1}
Suppose that the framework $({\mathcal F})$ together with extra conditions:
\[  T \in (0, \infty], \quad \nu(x)=\nu_{\star}~\mbox{(constant)}, \quad x \in \Omega, \quad 0 < {\mathcal D}_0 < \pi, \]
holds, and let $\theta^N$ and $\theta$ be solutions to \eqref{D-3} and \eqref{D-2} in the time interval $[0, T)$, respectively. Then, there exists some constant $\bar{C}_T$ independent of time $t$ such that
\[
\sup_{0 \leq t < T} \|\theta^N(t)-\theta(t)\|_{2}^2
\;\le\; \bar{C}_T \Big(
\|\theta^{\mathrm{in},N}-\theta^{\mathrm{in}}\|_{2}^2
+\|\psi^N-\psi\|_{1}
\Big).
\]
\end{corollary}
\begin{proof}
Recall that in the proof of Theorem \ref{T4.1} we applied Gr\"onwall’s lemma to \eqref{D-4} to obtain
	\[
	\bigl\|\theta^{N}(t)-\theta(t)\bigr\|_{2}^{2}
	\;\le\;
	C t e^{C t} \left( \bigl\|\theta^{N, \mathrm{in}}-\theta^{\mathrm{in}} \bigr\|_{2}^{2}+\bigl\|\nu^N-\nu\|_{1}+\|\psi^{N}-\psi\|_{1} \right).
	\]
	Thus, for any \(T>0\), we have
	\begin{align*}
		\begin{aligned}
			& \sup_{0\le t\le T}\bigl\|\theta^{N}(t)-\theta(t)\bigr\|_{2}^{2}  
 \leq \bar{C}_T \Big(\|\theta^{N, {\mathrm{in}}}-\theta^{\mathrm{in}} \|_{2}^{2}+\|\psi^{N}-\psi\|_{1}\Big),
		\end{aligned}
	\end{align*}
	where \(\bar{C}_T:= C T  e^{C T}\).
\end{proof}
Recall that in Theorem \ref{T3.2}, we obtained the complete phase synchronization:
\[\mathcal{D}(\theta(t))=0, \quad t \geq t_* = \frac{(1-\beta)\mathcal{D}_0^{1+\alpha} }{\alpha \kappa (\sin \mathcal{D}_0) (2-2^\beta)}. \]
This only depends on the diameter of the initial data, and we know 
\[\mathcal{D}_0\ge \mathcal{D}(\theta^{N, \mathrm{in}}),\quad \lim_{N\to\infty}\mathcal{D}(\theta^{N, \mathrm{in}})=\mathcal{D}_0.\] 
Therefore, we can find a positive constant $t_*^{\infty}$  such that
\[\mathcal{D}(\theta(t))=\mathcal{D}(\theta^{N}(t))=0, \quad t \geq t_*^{\infty}. \]
Note that 
\[\int_{\Omega} \theta(t,x)\di x= \int_{\Omega}\theta(0,x)\di x=\int_{\Omega}\theta^{N}(0,x)\di x=\int_{\Omega}\theta^{N}(t,x)\di x=0. \]
This implies
\[\theta(t,x)=\theta^{N}(t,x)=0, \quad t  \geq t_*^{\infty}, \]
which gives 
\[\|\theta^N(t)-\theta(t)\|_{2}^2=0, \quad t  \geq t_*^{\infty}. \]
Finally, we use Theorem \ref{T4.1} to find the uniform-in-time graph limit.

\section{Numerical simulations} \label{sec:5}
\setcounter{equation}{0}
In this section, we provide several numerical simulations for the constant and nonconstant natural frequency functions, respectively, and compare them with analytic results studied in previous sections. In all simulations, we have used spatial and phase regularization for stability of the discrete operator, implicit midpoint method and first-order explicit Euler for adaptive time-stepping and quadrature to discretize the nonlinear term in \eqref{A-2}. More precisely, we delineate the procedure as follows.

\begin{itemize}
\item
Step A: The spatial domain $[0,1]$ is discretized uniformly using
\begin{align*}
x_i = i/N, \quad i = 0, 1, \dots, N = 512.
\end{align*}
\item 
Step B: The kernel $\psi(x,y) = |x-y|^{-\beta}$ is discretized into a dense matrix 
\begin{align*}
[\Psi]_{ij} = \max \lbrace |x_i - x_j|, \varepsilon \rbrace^{-\beta},
\end{align*}
where the cutoff $\varepsilon = 10^{-9}$ regularizes the singularity near $x_i = x_j$.
\vspace{0.1cm}

\item
Step C: The right-hand side of $\eqref{A-2}_1$ is replaced by pairwise summation:
\begin{align} \label{E-0}
\frac{\di \theta_i}{\di t} = \nu_i + \kappa \sum_{j \neq i} \Psi_{ij} \frac{\sin(\theta_j - \theta_i)}{\max \lbrace |\theta_j - \theta_i|, \delta \rbrace^\alpha}
\end{align}
with small-angle regularization $\delta = 10^{-3}$.
\vspace{0.1cm}

\item
Step D: The time-stepping is based on the implicit midpoint method:
\begin{align} \label{E-1}
\theta^{n+1} = \theta^n + \Delta t f \left( \frac{\theta^n +\theta^{n+1}}{2} \right),
\end{align}
where $f$ is the discrete right-hand side in \eqref{E-0}. We solve this implicit equation using Picard iteration:
\begin{align*}
\theta^{(k+1)} = (1-\omega) \theta^{(k)} + \omega \left[ \theta^n + \Delta t f \left( \frac{\theta^n + \theta^{(k)}}{2} \right) \right]
\end{align*}
with relaxation parameter $\omega =0.7$.
\vspace{0.1cm}

\item
Step E: The time-step is adapted based on the local error estimate, which is done as follows: 
\newline
\begin{itemize}
\item Step E.1: Compute one step with the implicit midpoint method $\theta_\mathrm{new}$ as in \eqref{E-1}. \newline
\item Step E.2: Compute a first-order explicit Euler step:
\begin{align*}
\theta_\mathrm{explicit} = \theta^n + \Delta f(\theta^n). 
\end{align*}
\item Step E.3: Estimate local error as 
\begin{align*}
\mathrm{err} = \max_i |\theta_\mathrm{new,i} - \theta_\mathrm{explicit,i}|.
\end{align*}
\item Step E.4: Adjust the time step according to:
\begin{align*}
\Delta t_\mathrm{new} = 0.9 \times \Delta t \times  \left( \frac{10^{-4}}{\mathrm{err}} \right)^\frac{1}{2}.
\end{align*}
\end{itemize}
This ensures temporal accuracy and stability while avoiding unnecessary small steps.
\end{itemize}

\subsection{Homogeneous ensemble} \label{sec:5.1}
In this subsection, we consider a homogeneous ensemble in which the natural frequency function is simply constant. Then, without loss of generality, we may assume 
\[  \nu(x) = 0, \quad \forall~x \in \Omega =  [0,1]. \]
In this case, the SCKM \eqref{A-2} becomes 
\[ 
\begin{cases}
\displaystyle \partial_t \theta(t, x) =  \kappa \int_{0}^{1}   \frac{1}{|x-y|^{\beta}}  \frac{\sin( \theta(t, y) - \theta(t, x))}{|\theta(t,y) - \theta(t,x) |^{\alpha}_o} \di y, \quad t > 0,~~x \in [0, 1], \vspace{6pt}\\
\displaystyle \theta(0, x) = \sin  x, \quad x \in [0,1].
\end{cases}
\]
Next, we numerically investigate the finite-time formation of complete phase synchronization and compare the convergence speed depending on the relative sizes between $\alpha, \beta$ and $\kappa$. The numerical results yield the following observations which are consistent with our theoretical results in Theorem \ref{T3.2}.

\begin{itemize}
\item Increasing the parameter~$\beta$ leads to faster synchronization, as evidenced in Figures~\ref{fig1a} and~\ref{fig1b}.
\vspace{0.1cm}

\item Increasing the parameter~$\alpha$ leads to faster synchronization, as evidenced in Figures~\ref{fig2a} and~\ref{fig2b}.
\vspace{0.1cm}

\item Comparing Figures~\ref{fig1a} and~\ref{fig1b} with Figures~\ref{fig2a} and~\ref{fig2b}, respectively, shows that parameter $\alpha$ has a stronger effect on the synchronization rate than parameter $\beta$.
\vspace{0.1cm}

\item Increasing coupling strength $\kappa$ leads to faster synchronization, as can be seen in Figures~\ref{fig3a} and~\ref{fig3b}.
\end{itemize}

\begin{figure}
    \centering
    \begin{subfigure}[b]{0.45\textwidth} 
        \centering
        \includegraphics[width=\linewidth]{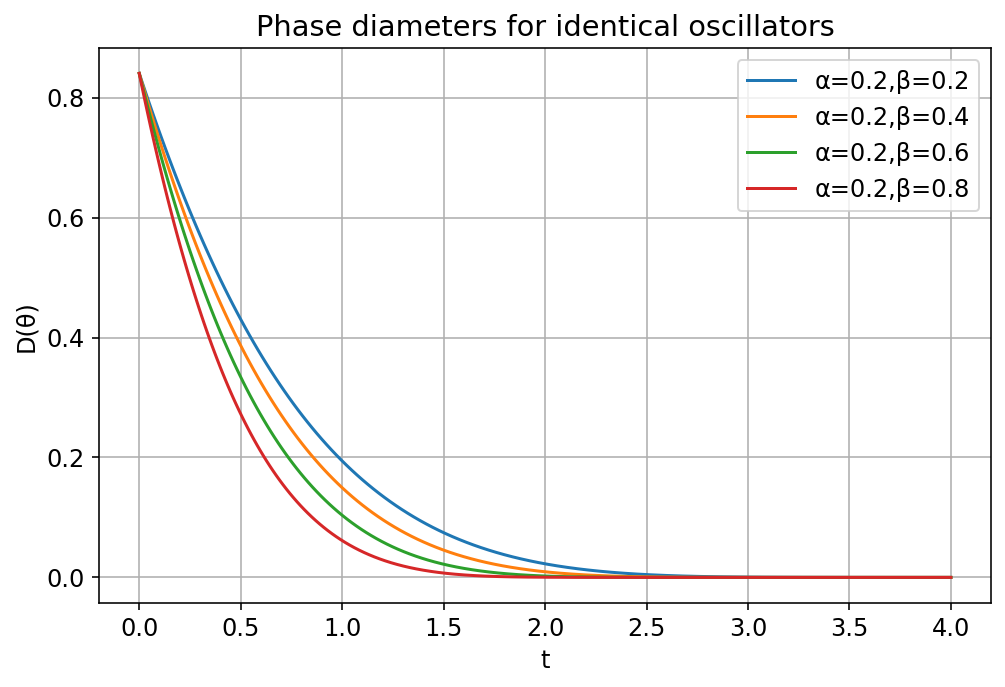}
        \caption{$\alpha = 0.2$, $\beta = 0.2,  0.4, 0.6, 0.8$, $\kappa = 1.0$}
        \label{fig1a}
    \end{subfigure}
    \hfill 
    \begin{subfigure}[b]{0.45\textwidth}
        \centering
        \includegraphics[width=\linewidth]{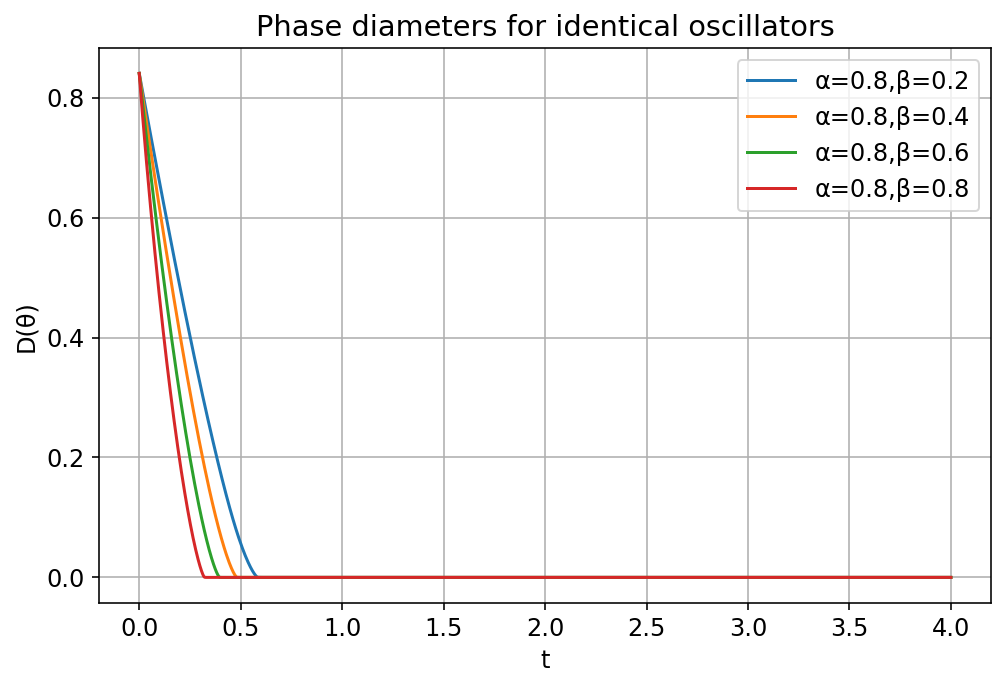}
        \caption{$\alpha = 0.8$, $\beta = 0.2,  0.4, 0.6, 0.8$, $\kappa = 1.0$}
        \label{fig1b}
    \end{subfigure}
        \caption{Comparison of phase diameter dynamics in homogeneous case for fixed $\alpha$ and varying kernel singularity parameter $\beta$.}
\end{figure}
    
    \begin{figure}
    \centering
    \begin{subfigure}[b]{0.45\textwidth} 
        \centering
        \includegraphics[width=\linewidth]{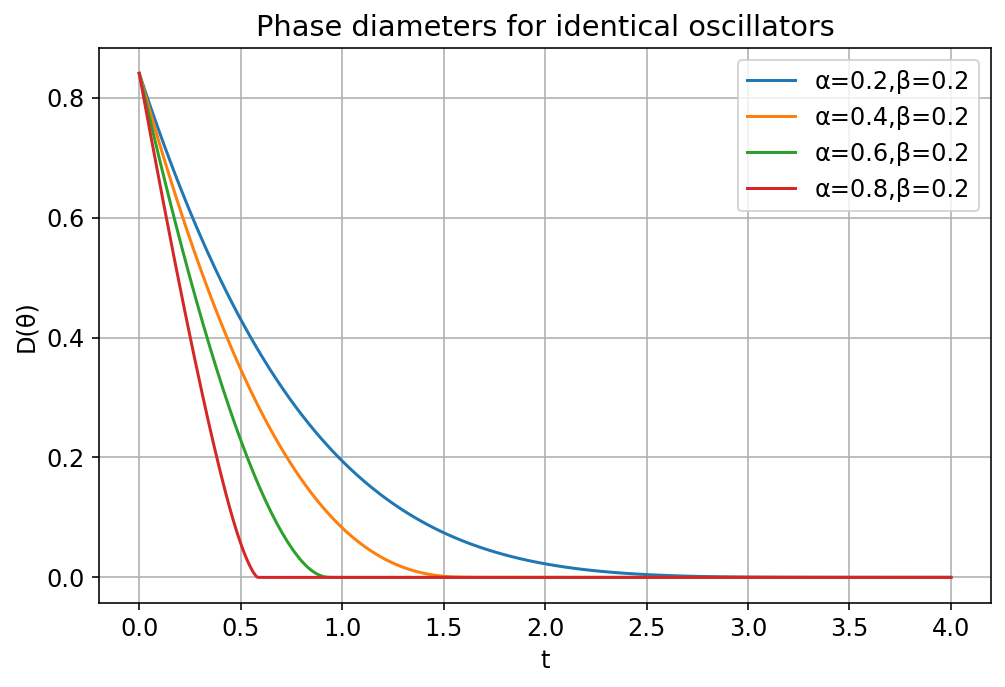}
        \caption{$\beta = 0.2$, $\alpha = 0.2,  0.4, 0.6, 0.8$, $\kappa = 1.0$}
        \label{fig2a}
    \end{subfigure}
    \hfill 
    \begin{subfigure}[b]{0.45\textwidth}
        \centering
        \includegraphics[width=\linewidth]{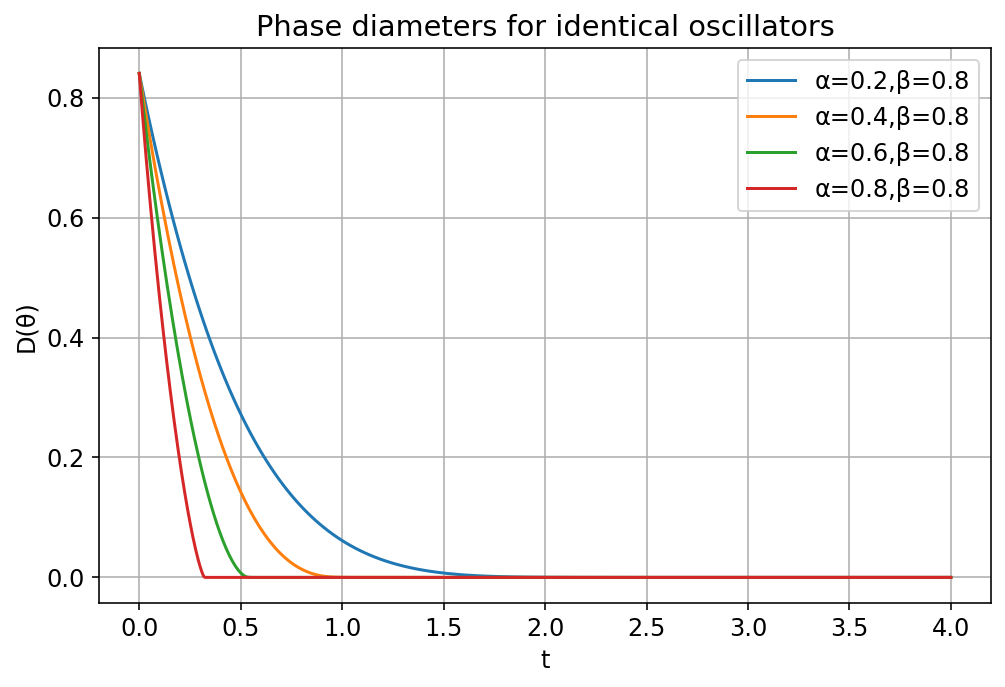}
        \caption{$\beta = 0.8$, $\alpha = 0.2,  0.4, 0.6, 0.8$, $\kappa = 1.0$}
        \label{fig2b}
    \end{subfigure}
    \caption{Comparison of phase diameter dynamics in homogeneous case for fixed $\beta$ and varying kernel singularity parameter $\alpha$.}
\end{figure}

\begin{figure}
    \centering
    \begin{subfigure}[b]{0.45\textwidth} 
        \centering
        \includegraphics[width=\linewidth]{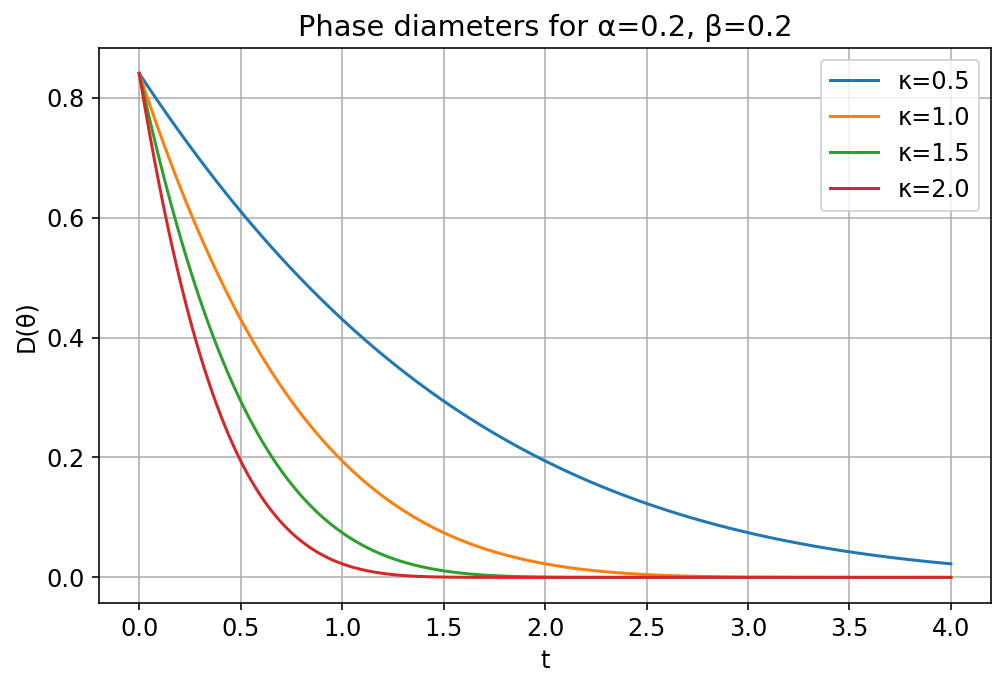}
        \caption{$\alpha = 0.2, \beta = 0.2$, $\kappa = 0.5, 1.0, 1.5, 2.0$}
        \label{fig3a}
    \end{subfigure}
    \hfill 
    \begin{subfigure}[b]{0.45\textwidth}
        \centering
        \includegraphics[width=\linewidth]{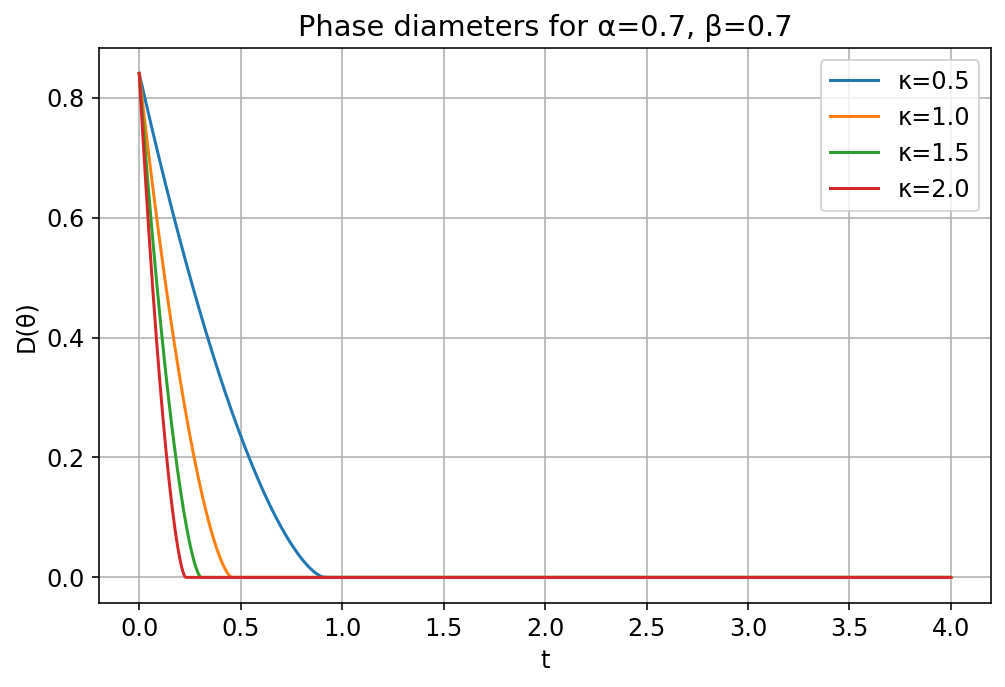}
        \caption{$\alpha = 0.6, \beta = 0.6$, $\kappa = 0.5, 1.0, 1.5, 2.0$}
        \label{fig3b}
    \end{subfigure}
    \caption{Comparison of phase diameter dynamics in homogeneous case for fixed $\alpha$ and $\beta$ and varying coupling strength $\kappa$.}
\end{figure}

\subsection{Heterogeneous ensemble} \label{sec:5.2}
In this subsection, we consider a heterogeneous ensemble in which the natural frequency function is as follows: 
\[  \nu(x) = \cos x, \quad \forall~x \in \Omega =  [0, 1]. \]
In this case, the SCKM \eqref{A-2} becomes 
\[ 
\begin{cases}
\displaystyle \partial_t \theta(t, x) = \cos x +  \kappa \int_{0}^{1}   \frac{1}{|x-y|^{\beta}}  \frac{\sin( \theta(t, y) - \theta(t, x))}{|\theta(t,y) - \theta(t,x) |^{\alpha}_o} \di y, \quad t > 0,~~x \in [0, 1], \vspace{6pt}\\
\displaystyle \theta(0, x) = \sin  x, \quad x \in [0,1].
\end{cases}
\]
Next, we numerically investigate the practical synchronization in Theorem~\ref{T3.3} and compare the convergence speed depending on the relative sizes between $\alpha, \beta$ and $\kappa$. We compute the critical coupling strength $\kappa_*$ according to \eqref{C-7-1} for $\varepsilon = 0.5$ in Theorem~\ref{T3.3}, and we consider two values of $\kappa$ below and two above this critical threshold. We make the following observations:

\begin{itemize}
\item In the case~$\alpha = \beta = 0.2$, the critical coupling strength is computed as~$\kappa_* = 0.938074$. As shown in Figure~\ref{fig4a}, when~$\kappa$ is below this critical value, the phase diameter exceeds the threshold~$\varepsilon = 0.5$. In contrast, for~$\kappa$ values greater than the critical coupling strength, the phase diameter remains strictly below this threshold.
\vspace{0.1cm}

\item In the case~$\alpha = \beta = 0.7$, the critical coupling strength is computed as~$\kappa_* = 0.731129$. In Figure~\ref{fig4b}, we observe that higher singularity parameters, $\alpha = \beta = 0.7$, contribute to maintaining the phase diameter below the threshold~$\varepsilon = 0.5$, in some cases even when~$\kappa$ is below the theoretically computed critical coupling strength. When~$\kappa$ exceeds the critical value, the phase diameter is significantly smaller than for weaker parameters $\alpha = \beta=0.2$.
\end{itemize}

 \begin{figure}
    \centering
    \begin{subfigure}[b]{0.45\textwidth} 
        \centering
        \includegraphics[width=\linewidth]{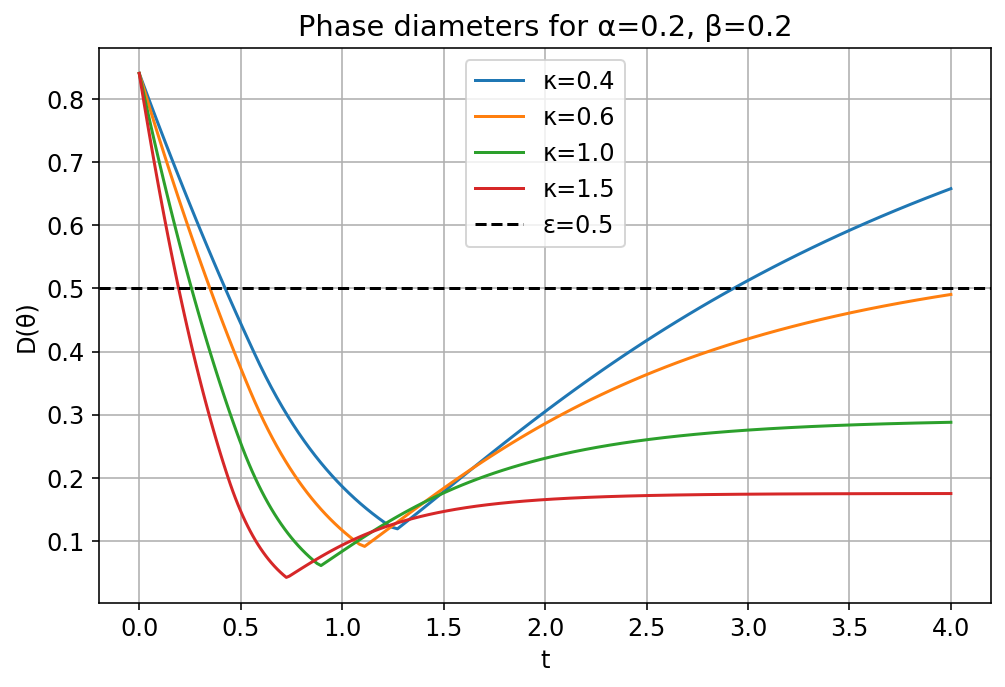}
        \caption{$\alpha = 0.2, \beta = 0.2$, $\varepsilon = 0.5$, \newline $\kappa = 0.4, 0.6, 1.0, 1.5$}
        \label{fig4a}
    \end{subfigure}
    \hfill  
    \begin{subfigure}[b]{0.45\textwidth}
        \centering
        \includegraphics[width=\linewidth]{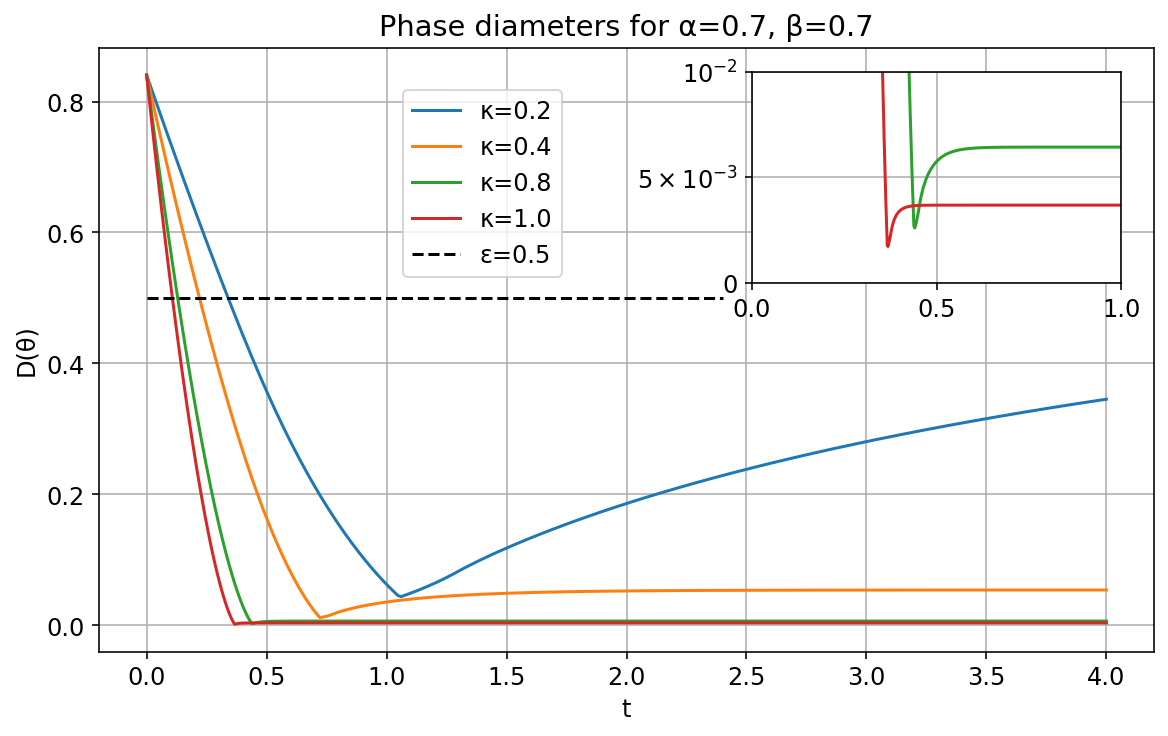}
        \caption{$\alpha = 0.8, \beta = 0.8$, $\varepsilon = 0.5$, \newline $\kappa = 0.2, 0.4, 0.8, 1.0$}
        \label{fig4b}
    \end{subfigure}
    \caption{Comparison of phase diameter dynamics in heterogeneous case for fixed $\alpha$ and $\beta$ and varying coupling strength $\kappa$.}
\end{figure}

\section{Conclusion} \label{sec:6}
In this paper, we have studied the emergent dynamics and graph limit of the SCKM under the double singular effects. Our results showed that the singular kernel will lead to a sticking effect in finite time, which is very different from regular kernels. In particular, we have provided frameworks for the complete synchronization and practical synchronization under homogeneous and heterogeneous natural frequencies, respectively. On the other hand, we established the rigorous graph limit from the finite particle model to the SCKM. This yields that the SCKM is the natural description of a large scale of the ensemble of infinite Kuramoto oscillators with singular communication weight and couplings. Moreover, there are several issues to be explored in future. To name a few, first whether the singular integro-differential equation admits a gradient flow structure which allows us to show stronger synchronization results. Second, a rigorous collective dynamics analysis of singular stochastic continuous integro-differential equation will be a very interesting question. 

\section*{Conflict of interest statement}
The authors declare no conflicts of interest.

\section*{Data availability statement}
The data supporting the findings of this study are available from the corresponding author upon reasonable request.

\appendix

\section{Proofs of basic results} \label{App-0}
\setcounter{equation}{0}
In this appendix, we present the proofs of some basic results in Section \ref{sec:2}.

\begin{proof}[Proof of Lemma~\ref{L2.2}]
Since the map $u\mapsto u^q$ is concave on $[0,\infty)$, we have
\begin{align*}
u^q  = \left( \frac{u}{u+v} (u+v) + \frac{v}{u+v} \cdot 0 \right)^q \geq \frac{u}{u+v} (u+v)^q + \frac{v}{u+v} 0^q = \frac{u}{u+v} (u+v)^q
\end{align*}
and
\begin{align*}
v^q  = \left( \frac{v}{u+v} (u+v) + \frac{u}{u+v} \cdot 0 \right)^q \geq \frac{v}{u+v} (u+v)^q + \frac{u}{u+v} 0^q = \frac{v}{u+v} (u+v)^q
\end{align*}
for all $u,v\geq 0$. We add the above two equations to get
\begin{align*}
u^q + v^q \geq (u+v)^q.
\end{align*}
\end{proof}

\begin{proof}[Proof of Lemma \ref{L2.4}]
\noindent (i)~Note that 
\begin{align*}
g'(\theta) = \frac{ \theta^{1+\alpha } \cos (\theta)- (1+\alpha) \theta^\alpha \sin (\theta)}{\theta^{2+ 2\alpha}} = \frac{\theta\cos(\theta)  - (1+\alpha)  \sin (\theta)}{ \theta^{2+ \alpha}}.
\end{align*}
This implies the following equivalence relations: for $\theta \in (0, \pi)$, 
\begin{equation*} \label{B-2-5-3}
g^{\prime}(\theta) \leq 0 \quad \iff \quad \theta \cos (\theta) - (1+\alpha)  \sin (\theta) \leq 0 \quad \iff \quad  1+\alpha \geq \theta \cot (\theta).
\end{equation*}
Next, we analyze the function $\theta \cot (\theta)$ on the intervals $\left[ \frac{\pi}{2}, \pi \right)$ and $(0,\frac{\pi}{2})$, respectively. \newline

\noindent $\bullet$~Case A $( \theta \in \left[ \frac{\pi}{2}, \pi \right))$:~In this case, the function $\theta \mapsto \cot (\theta)$ is nonpositive on $\left[ \frac{\pi}{2}, \pi \right)$. Hence, we have 
\[ 1+\alpha \geq \theta \cot (\theta) \quad \mbox{on $\left[ \frac{\pi}{2}, \pi \right)$}. \]
\noindent $\bullet$~Case B $(\theta \in (0,\frac{\pi}{2}))$:~By direct calculation, one has 
\begin{align*}
(\theta \cot (\theta))'=\left( \frac{\theta\cos(\theta)}{\sin (\theta)} \right)' = \frac{(\cos (\theta) - \theta\sin (\theta)) \sin (\theta) - \theta\cos^2 (\theta)}{\sin^2 (\theta)} = \frac{\frac{1}{2} \sin (2\theta) - \theta}{\sin^2 (\theta)} \leq 0,
\end{align*} 
so the function $\theta \mapsto \theta \cot (\theta)$ is nonincreasing. Moreover, we have
\begin{align*}
\lim_{\theta \rightarrow 0} \frac{\theta \cos (\theta)}{\sin (\theta)} = \lim_{\theta \rightarrow 0} \frac{\cos (\theta) - \theta\sin (\theta)}{\cos (\theta)} =1.
\end{align*}
This implies 
\[\sup\limits_{\theta \in(0,\frac{\pi}{2})} \theta \cot (\theta) \le1 \leq 1 + \alpha. \]  
Finally, we combine the results of Case A and Case B to derive the desired estimate. \\

\noindent (ii)~Note that
\begin{align*}
f' (\theta) = \frac{\theta \cos (\theta)- \sin (\theta)}{\theta^2}
\end{align*}
and
\begin{align*}
\frac{\di}{\di \theta} (\theta \cos (\theta)- \sin (\theta)) = -\theta \sin (\theta)  + \cos (\theta) - \cos (\theta)= -\theta \sin (\theta) <0, \quad \mathrm{for} \quad \theta \in (0,\pi).
\end{align*}
Therefore, we have the desired estimate:
\begin{align*}
f' (\theta) \leq \lim_{z\rightarrow 0+}f'(z)=0 \quad \mathrm{on} \quad (0,\pi).
\end{align*}
\end{proof}

\section{Proof of Theorem \ref{T3.1}} \label{App-A}
In this appendix, we present detailed arguments outlined in the proof of Theorem \ref{T3.1}. \newline

\noindent$\bullet$~Step A (Extraction of a convergent subsequence): In the sequel, we derive equicontinuity and relative compactness of $(\theta_{\varepsilon})_{\varepsilon>0}$ separately. \\

\noindent$\diamond$~Step A.1 (Equicontinuity):~For every $s,t\in [0,T]$ we have
\begin{align} 
\begin{aligned} \label{App-A-2}
&\|\theta_\varepsilon(t) - \theta_\varepsilon(s)\|_2^2   = \int_\Omega |\theta_\varepsilon(t,x) - \theta_\varepsilon(s,x)|^2 \di x \\
&\hspace{0.2cm}= \int_\Omega \left( \nu (x)(t-s) +  \int_s^t \kappa \int_\Omega  \psi_\varepsilon(x,y)h_\varepsilon(\theta_\varepsilon(\tau,y)- \theta_\varepsilon(\tau,x)) \di y \di \tau  \right)^2 \di x \\
& \hspace{0.2cm} \leq  2|t-s|^2 \|\nu \|_2^2 + 2\int_\Omega \left( \int_s^t \kappa \int_\Omega  \psi_\varepsilon(x,y) \di y \di \tau  \right)^2 \di x\\
& \hspace{0.2cm} \leq  2\left(  \|\nu \|_2^2 + C_\psi^2 \kappa^2 \right) |t-s|^2 .
\end{aligned}
\end{align}
This gives us a Lipschitz constant which does not depend on $\varepsilon$. \\

\noindent$\diamond$~Step A.2 (Relative compactness):~We now show that the sequence $(\theta_\varepsilon(t, \cdot))_{\varepsilon >0}$ is equicontinuous for every fixed $t\in [0,T]$. Since we consider all angles in $(-\pi,\pi]$, for every $v>0$ it holds
\begin{align} \label{App-A-2-1}
\begin{aligned}
&\frac{\di}{\di t}\| \theta_\varepsilon (t,\cdot-v) - \theta_\varepsilon (t,\cdot) \|_2^2 = 2 \int_\Omega (\theta_\varepsilon (t,x-v) - \theta_\varepsilon (t,x))(\partial_t \theta_\varepsilon (t,x-v) - \partial_t \theta_\varepsilon (t,x)) \di x \\
& \hspace{0.5cm} = 2 \int_\Omega (\theta_\varepsilon (t,x-v) - \theta_\varepsilon (t,x))(\nu(x-v) - \nu(x)) \di x  \\
& \hspace{0.7cm}+ 2\kappa\int_\Omega (\theta_\varepsilon (t,x-v) - \theta_\varepsilon (t,x))   \int_\Omega \psi_\varepsilon (x-v,y) h_\varepsilon (\theta_\varepsilon (t,y) - \theta_\varepsilon (t,x-v)) \di y  \di x \\
& \hspace{0.7cm}- 2\kappa\int_\Omega (\theta_\varepsilon (t,x-v) - \theta_\varepsilon (t,x))  \int_\Omega \psi_\varepsilon (x,y) h_\varepsilon (\theta_\varepsilon (t,y) - \theta_\varepsilon (t,x)) \di y \di  x \\
& \hspace{0.5cm} \leq  4\pi \| \nu(\cdot-v) - \nu(\cdot) \|_1 \\
& \hspace{0.7cm}+ 4\pi\kappa \iint_{\Omega^2} \left| (\psi_\varepsilon (x-v,y)- \psi_\varepsilon (x,y)) h_\varepsilon (\theta_\varepsilon (t,y) - \theta_\varepsilon (t,x-v)) \right| \di y \di x \\
& \hspace{0.7cm}+ 2\kappa\int_\Omega (\theta_\varepsilon (t,x-v) - \theta_\varepsilon (t,x))  \\
& \hspace{1cm}\times \int_\Omega \psi_\varepsilon (x,y) (h_\varepsilon (\theta_\varepsilon (t,y) - \theta_\varepsilon (t,x-v)) -h_\varepsilon (\theta_\varepsilon (t,y) - \theta_\varepsilon (t,x)) ) \di y \di  x  \\
& \hspace{0.5cm} =: {\mathcal J}_{11} + {\mathcal J}_{12}+ {\mathcal J}_{13}.
\end{aligned}
\end{align}
In the sequel, we estimate the terms ${\mathcal J}_{1i},~i=1,\cdots, 3$ one by one.\\

\noindent $\clubsuit$~Estimate of $\mathcal{J}_{11}$:~Using $\nu \in L^2(\Omega)$ and Lemma \ref{L2.13}, we have
\begin{align} \label{App-A-2-2}
\lim_{v \to 0} \mathcal{J}_{11} = 0.
\end{align}
Note that the term $\mathcal{J}_{11}$ does not depend on $\varepsilon$. \\
 
\noindent $\clubsuit$~Estimate of $\mathcal{J}_{12}$: It holds
\begin{align*} 
\begin{aligned}
\mathcal{J}_{12} &= 4\pi\kappa\iint_{\Omega^2} \left| (\psi_\varepsilon (x-v,y)- \psi_\varepsilon (x,y)) h_\varepsilon (\theta_\varepsilon (t,y) - \theta_\varepsilon (t,x-v)) \right| \di y \di x \\
& \leq 4\pi \kappa \sup_{x\in \Omega} \int_\Omega |\psi_\varepsilon (x-v,y)- \psi_\varepsilon (x,y)| \di y.
\end{aligned}
\end{align*}
Applying Lemma \ref{L2.5}, it follows 
\begin{align} \label{App-A-2-4}
\lim_{v \to 0} \mathcal{J}_{12} = 0 \quad \mathrm{uniformly \ in} \ \varepsilon.
\end{align}\\

\noindent $\clubsuit$~Estimate of $\mathcal{J}_{13}$: We use one-sided Lipschitz continuity of $-h_\varepsilon$ from Lemma \ref{L2.6-1} and $-h_\varepsilon(z) = h_\varepsilon(-z)$ to obtain
\begin{align} 
&\mathcal{J}_{13} = 2\kappa\int_\Omega (\theta_\varepsilon (t,x-v) - \theta_\varepsilon (t,x)) \nonumber \\
& \hspace{1cm}\times \int_\Omega \psi_\varepsilon (x,y) (h_\varepsilon (\theta_\varepsilon (t,y) - \theta_\varepsilon (t,x-v)) -h_\varepsilon (\theta_\varepsilon (t,y) - \theta_\varepsilon (t,x)) ) \di y  \di  x\nonumber \\
&\hspace{0.7cm}\leq  2\kappa\int_\Omega (\theta_\varepsilon (t,x-v) - \theta_\varepsilon (t,y) - (\theta_\varepsilon (t,x) - \theta_\varepsilon (t,y)))  \nonumber\\
& \hspace{1cm}\times \int_\Omega \psi_\varepsilon (x,y) ((-h_\varepsilon) (\theta_\varepsilon (t,x-v) - \theta_\varepsilon (t,y)) -(-h_\varepsilon) (\theta_\varepsilon (t,x) - \theta_\varepsilon (t,y)) ) \di y \di  x \nonumber\\
&\hspace{0.7cm}\leq 2 L\kappa \iint_{\Omega^2} (\theta_\varepsilon (t,x-v) - \theta_\varepsilon (t,x))^2 \psi_\varepsilon (x,y) \di y \di x \nonumber \\ \label{App-A-2-5}
&\hspace{0.7cm}\leq 2C_\psi L\kappa \| \theta_\varepsilon (t,\cdot-v) - \theta_\varepsilon (t,\cdot) \|_2^2.
\end{align}
Applying Gr\"onwall's lemma to \eqref{App-A-2-1} and using \eqref{App-A-2-5} yields 
\begin{align} \label{App-A-2-6}
\| \theta_\varepsilon (t,\cdot-v) - \theta_\varepsilon (t,\cdot) \|_2^2 \leq \left( \| \theta^\mathrm{in} (\cdot-v) - \theta^\mathrm{in} (\cdot) \|_2^2  + \mathcal{J}_{21}t + \mathcal{J}_{22}t \right) e^{2C_\psi L \kappa t}.
\end{align}
Using $\theta^\mathrm{in} \in L^2 (\Omega)$ and Lemma \ref{L2.13}, we get
\begin{align} \label{App-A-2-7}
\lim_{v\to 0}\| \theta^\mathrm{in} (\cdot-v) - \theta^\mathrm{in} (\cdot) \|_2^2 =0.
\end{align}
Combining \eqref{App-A-2-6} with \eqref{App-A-2-7}, \eqref{App-A-2-2} and \eqref{App-A-2-4} gives us the desired result. \\

By Kolmogorov-Riesz-Fr\'echet's theorem (see Proposition \ref{P3.2}),  the subset $(\theta_\varepsilon(t, \cdot))_{\varepsilon >0}$ of $L^2(\Omega)$ is relatively compact. Hence, we can apply the generalized theorem of Arzel\`a-Ascoli and extract a subsequence $(\theta_{\varepsilon_n})_{n\in \mathbb{N}}$, where $\varepsilon_n \rightarrow 0$ for $n\rightarrow \infty$, and find some $\theta \in C ([0,T]; L^2(\Omega))$ such that
\begin{align} \label{App-A-4}
\sup_{t\in [0,T]} \|\theta_{\varepsilon_n} (t) - \theta(t)\|_2 \rightarrow 0, \quad \mbox{as $n \to \infty$}.
\end{align}
\noindent$\bullet$~Step B (Limit solves the Cauchy problem):~We proceed by showing that $\theta$ is a solution to the SCKM \eqref{A-2}. We claim that
\begin{align}
\begin{aligned} \label{App-A-5}
\sup_{t\in [0,T]} \Bigg\| \kappa \int_0^t \int_\Omega \psi_\varepsilon  &(x,y)  h_\varepsilon (\theta_\varepsilon (\tau, y) - \theta_\varepsilon (\tau,x)) \di y \di \tau\\
 - \kappa \int_0^t &\int_\Omega \psi (x,y) h(\theta (\tau, y) - \theta (\tau,x)) \di y \di \tau \Bigg\|_2 \underset{\varepsilon \rightarrow 0}{\longrightarrow} 0,
\end{aligned}
\end{align}
where we omit the index $n\in\mathbb{N}$ in the subsequence $(\varepsilon_n)_{n\in \mathbb{N}}$ for simplicity. \\

{\it Proof of \eqref{App-A-5}}:~Note that
\begin{align}
&\int_\Omega \left( \int_0^t \int_\Omega \psi_\varepsilon (x,y) h_\varepsilon (\theta_\varepsilon (\tau, y) - \theta_\varepsilon (\tau,x) ) \di y \di \tau - \int_0^t \int_\Omega \psi (x,y) h(\theta (\tau, y) - \theta (\tau,x)) \di y \di \tau \right)^2   \di x \nonumber \\
& \hspace{0.2cm} = \int_\Omega \bigg( \int_0^t \int_\Omega (\psi_\varepsilon (x,y) - \psi(x,y)) h_\varepsilon(\theta_\varepsilon (\tau, y) - \theta_\varepsilon (\tau,x)) \di y \di \tau \nonumber \\
& \hspace{0.4cm} + \int_0^t \int_\Omega \psi (x,y) (h_\varepsilon(\theta_\varepsilon (\tau, y) - \theta_\varepsilon (\tau,x)) - h_\varepsilon(\theta(\tau, y) - \theta (\tau,x))) \di y \di \tau  \nonumber \\
& \hspace{0.4cm} + \int_0^t \int_\Omega \psi (x,y) (h_\varepsilon(\theta(\tau, y) - \theta(\tau,x)) - h(\theta(\tau, y) - \theta (\tau,x))) \di y \di \tau \bigg)^2   \di x  \nonumber \\
& \hspace{0.2cm} \leq 3t^2 \int_\Omega \left( \int_\Omega |\psi_\varepsilon (x,y) - \psi(x,y)|  \di y \right)^2 \di x \nonumber \\
& \hspace{0.4cm} + 3 \int_\Omega \left( \int_0^t \int_\Omega \psi (x,y) (h_\varepsilon(\theta_\varepsilon (\tau, y) - \theta_\varepsilon (\tau,x)) - h_\varepsilon(\theta (\tau, y) - \theta (\tau,x))) \di y \di \tau  \right)^2 \di x \nonumber \\
& \hspace{0.4cm} + 3 \int_\Omega \left( \int_0^t \int_\Omega \psi (x,y) (h_\varepsilon(\theta(\tau, y) - \theta(\tau,x)) - h(\theta(\tau, y) - \theta (\tau,x))) \di y \di \tau \right)^2 \di x \nonumber \\ \label{App-A-5-1}
& \hspace{0.2cm} =: \mathcal{J}_{21}+\mathcal{J}_{22} +\mathcal{J}_{23}.
\end{align}
Below, we estimate ${\mathcal J}_{2i},~i =1,2,3$ one by one. \\

\noindent $\clubsuit$ Estimate of $\mathcal{J}_{21}$:~We make use of Lemma \ref{L2.5} and Lebesgue dominated convergence theorem to find 
\begin{equation} \label{App-A-11}
\mathcal{J}_{21} \to 0 \quad \mbox{as $\varepsilon \to 0$}.
\end{equation} \\

\noindent $\clubsuit$ Estimate of $\mathcal{J}_{22}$:~We use Lemma \ref{L2.4-1} and Lemma \ref{L2.2} to obtain
\begin{align}
\begin{aligned} \label{App-A-7}
\mathcal{J}_{22} &=  3 \int_\Omega \left( \int_0^t \int_\Omega \psi (x,y) (h_\varepsilon(\theta_\varepsilon (\tau, y) - \theta_\varepsilon (\tau,x)) - h_\varepsilon(\theta (\tau, y) - \theta (\tau,x))) \di y \di \tau  \right)^2 \di x \\
&\leq  3\int_\Omega \left( C_h \int_0^t \int_\Omega \psi (x,y) (|\theta_\varepsilon (\tau, y) - \theta (\tau,y)|^{1-\alpha} +|\theta_\varepsilon (\tau, x) - \theta (\tau,x)|^{1-\alpha}) \di y \di \tau  \right)^2 \di x\\
&\leq  6C_h^2 \int_\Omega \left( \int_0^t \int_\Omega \psi (x,y)|\theta_\varepsilon (\tau, y) - \theta (\tau,y)|^{1-\alpha} \di y \di \tau  \right)^2 \di x \\
& \hspace{0.2cm} + 6C_h^2 \int_\Omega \left( \int_0^t \int_\Omega \psi (x,y)|\theta_\varepsilon (\tau, x) - \theta (\tau,x)|^{1-\alpha} \di y \di \tau  \right)^2 \di x \\
& =: \mathcal{J}_{221}+ \mathcal{J}_{222}.
\end{aligned}
\end{align}
\noindent $\spadesuit$ Estimate of $\mathcal{J}_{221}$: We use H\"older's inequality and Young's convolution inequality to see
\begin{align}
\begin{aligned} \label{App-A-8}
\mathcal{J}_{221} \leq  & 6 C_h^2 t \int_0^t \int_\Omega  \left( \int_\Omega \psi (x,y)|\theta_\varepsilon (\tau, y) - \theta (\tau,y)|^{1-\alpha} \di y \right)^2  \di x \di \tau \\
\leq &   6 C_h^2 t \int_0^t \| \Phi\|_1^2 \|\theta_\varepsilon (\tau) - \theta (\tau)\|_{2(1-\alpha)}^{2(1-\alpha)} \di \tau \leq 6C_\Phi^2 C_h^2 t^2 \sup_{t\in[0,T]} \|\theta_\varepsilon (t) - \theta (t)\|^{2(1-\alpha )}_2,
\end{aligned}
\end{align}
where $\Phi$ and $C_\Phi$ are defined as in \eqref{B-11-1} and \eqref{B-11-2}, respectively. In \eqref{App-A-8}, letting $\varepsilon \to 0$, we use \eqref{App-A-4} to see
\begin{equation} \label{App-A-9}
\mathcal{J}_{221} \to 0 \quad \mbox{as $\varepsilon \to 0$}.
\end{equation}

\noindent $\spadesuit$ Estimate of $\mathcal{J}_{222}$: We see
\begin{align}
\begin{aligned} \label{App-A-8-1}
\mathcal{J}_{222} \leq & 6C_\psi^2 C_h^2 t \int_0^t \int_\Omega |\theta_\varepsilon (\tau, x) - \theta (\tau,x)|^{2(1-\alpha)} \di x \di \tau \leq 6 C_\psi^2 C_h^2 t^2 \sup_{t\in[0,T]} \|\theta_\varepsilon (t) - \theta (t) \|_2^{2(1-\alpha)}. 
\end{aligned}
\end{align}
In \eqref{App-A-8-1}, letting $\varepsilon \to 0$, we use \eqref{App-A-4} to see
\begin{equation} \label{App-A-9-1}
\mathcal{J}_{222} \to 0 \quad \mbox{as $\varepsilon \to 0$}.
\end{equation}
Putting together \eqref{App-A-7}, \eqref{App-A-9} and \eqref{App-A-9-1}, we see
\begin{align} \label{App-A-9-2}
\mathcal{J}_{22} \to 0 \quad \mbox{as $\varepsilon \to 0$}.
\end{align}

\noindent $\clubsuit$ Estimate of $\mathcal{J}_{23}$: Note that 
\begin{align*}
\begin{aligned} 
\mathcal{J}_{23} \leq 3t^2 \sup_{t\in[0,T]} \int_\Omega \left( \int_\Omega \psi (x,y) (h_\varepsilon(\theta(t, y) - \theta(t,x)) - h(\theta(t, y) - \theta (t,x))) \di y  \right)^2  \di x.
\end{aligned}
\end{align*}
Using the Lebesgue dominated convergence theorem, we find 
\begin{align*}
\int_\Omega \psi (x,y) (h_\varepsilon(\theta(\tau, y) - \theta(\tau,x)) - h(\theta(\tau, y) - \theta (\tau,x))) \di y \to 0 \quad \mbox{as $\varepsilon \to 0$}.
\end{align*}
Since it holds
\begin{align*}
\left( \int_\Omega \psi (x,y) (h_\varepsilon(\theta(t, y) - \theta(t,x)) - h(\theta(t, y) - \theta (t,x)))  \di y \right)^2 \leq 4C_\psi^2,
\end{align*}
we can apply the Lebesgue dominated convergence theorem again to obtain
\begin{align} \label{App-A-13}
\mathcal{J}_{23} \to 0 \quad \mbox{as $\varepsilon \to 0$}.
\end{align}
In \eqref{App-A-5-1}, we use \eqref{App-A-11}, \eqref{App-A-9-2} and \eqref{App-A-13} to verify \eqref{App-A-5}. This completes the proof of Theorem \ref{T3.1}.

\bibliographystyle{amsplain}

\end{document}